\documentclass{mystyle}
\usetikzlibrary{automata}
\newcommand{\FS}{\mathsf{FS}}

\title[Friends-and-strangers graphs on complete multipartite graphs]{The connectivity of friends-and-strangers graphs on complete multipartite graphs}

\author[H. Zhu]{Honglin Zhu}

\begin{document}

\begin{abstract}
For simple graphs $X$ and $Y$ on $n$ vertices, the friends-and-strangers graph $\mathsf{FS}(X,Y)$ is the graph whose vertex set consists of all bijections $\sigma: V(X) \to V(Y)$, where two bijections $\sigma$ and $\sigma'$ are adjacent if and only if they agree on all but two adjacent vertices $a, b \in V(X)$ such that $\sigma(a), \sigma(b) \in V(Y)$ are adjacent in $Y$. Resolving a conjecture of Wang, Lu, and Chen, we completely characterize the connectedness of $\mathsf{FS}(X, Y)$ when $Y$ is a complete bipartite graph. We further extend this result to when $Y$ is a complete multipartite graph. We also determine when $\mathsf{FS}(X, Y)$ has exactly two connected components where $X$ is bipartite and $Y$ is a complete bipartite graph. 
\end{abstract}

\maketitle

\section{Introduction}\label{sec_intro}
Throughout this paper, we assume that all graphs are simple unless specified otherwise. The notion of a friends-and-strangers graph, defined by Defant and Kravitz in \cite{defant2021}, is given as follows.
\begin{definition}[\cite{defant2021}]
    Let $X$ and $Y$ be graphs on $n$ vertices. The \blue{\emph{friends-and-strangers}} graph on $X$ and $Y$, denoted $\FS(X, Y)$, is the graph whose vertex set is the set of bijections $\sigma: V(X) \to V(Y)$, and two bijections $\sigma$ and $\sigma'$ are adjacent in $\FS(X, Y)$ if and only if there exists an edge $\{a, b\} \in E(X)$ such that:
    \begin{itemize}
        \item $\{\sigma(a), \sigma(b)\} \in E(Y)$;
        \item $\sigma(a) = \sigma'(b)$, $\sigma(b) = \sigma'(a)$, and $\sigma(c) = \sigma'(c)$ for all $c \in V(X) \setminus \{a, b\}$.
    \end{itemize}
    The operation transforming $\sigma$ to $\sigma'$ is referred to as an \blue{\emph{$(X, Y)$-friendly swap}} across $\{a, b\}$. 
\end{definition}

The friends-and-strangers graph can be seen as a formalization of the following problem. Let $X$ be a graph on $n$ vertices. Suppose that $n$ people, who can pairwise be friends or strangers, with their friendships indicated by the edges of the graph $Y$, stand so that one person is at each vertex of $X$. At any time, two friends standing at adjacent vertices of $X$ may switch places. Our goal is to understand which configurations of people standing on this graph can be reached from others when allowing a sequence of these swaps. 

In the literature, people have often studied the structure of the friends-and-strangers graph $\FS(X, Y)$ when one of $X, Y$ is a specific graph (see \cite{wilson1974, defant2022, lee2022, jeong2022, wang2023, wang2023b, brunck2023}). Another type of problem one can explore is the structure of friends-and-strangers graphs on random graphs (see \cite{alon2023, milojevic2022, wang2023c}). Yet another possibility is to investigate $\FS(X, Y)$ when $X, Y$ satisfy certain properties but are not specific graphs (see \cite{alon2023,bangachev2022,defant2021,defant2022,jeong2023}). This paper falls into the first category, where we fix $Y$ to be a complete multipartite graph. 

Before Defant and Kravitz defined friends-and-strangers graphs, Wilson studied a special case of this setup in \cite{wilson1974} in the form of block puzzles. In the famous $15$-puzzle, we are given a $4 \times 4$ grid with $15$ numbered tiles occupying all but one of the squares. The goal is to slide those tiles to obtain the configuration where the tiles are ordered. This problem can be rephrased in terms of friends-and-strangers graphs; namely, it concerns the friends-and-strangers graph of the $4 \times 4$ grid graph and the star graph on $16$ vertices. Indeed, we can imagine $16$ people standing on the grid, with one person (the empty spot) who is a friend with everyone else, and there are no other friendships. In \cite{wilson1974}, Wilson computed the the number of connected components of $\FS(X, \mathsf{Star}_{n})$, where $X$ is biconnected (connected and does not have a cut-vertex) and $\mathsf{Star}_{n}$ is the star graph on $n$ vertices. 

As $\mathsf{Star}_{n}$ can also be seen as the complete bipartite graph $K_{1, n-1}$ with vertex bipartition into a singleton set and a set of size $n - 1$, Defant and Kravitz suggested in \cite{defant2021} that a natural extension of Wilson's result is to investigate the connectedness of $\FS(X, K_{k, n-k})$ for any graph $X$. Wang, Lu, and Chen (\cite{wang2023}) completely characterized when this graph is connected in the case $k = 2$, and also conjectured what happens when $k > 2$. In this paper, we prove their conjecture. To formulate the theorem, we need the following key definition.
\begin{definition}[\cite{wang2023}]\label{def_bridge}
    A path $v_1, v_2, \ldots, v_k$ in a graph is a \blue{\emph{$k$-bridge}} if each edge in the path is a cut-edge, $v_2, \ldots, v_{k - 1}$ have degree $2$ in the graph, and $v_1$ and $v_k$ do not have degree $1$. In particular, a non-trivial cut-edge (neither of its ends has degree $1$) is a $2$-bridge. We also say that a single cut-vertex is a $1$-bridge.
\end{definition}

Besides proving the conjecture of Wang, Lu, and Chen, we consider a further generalization to $\FS(X, K_{k_1, \ldots, k_t})$, where $K_{k_1, \ldots, k_t}$ is the complete $t$-partite graph with partition classes of sizes sizes $1 \leq k_1 \leq \cdots \leq k_t$. (Recall that two vertices in a complete multipartite graph are adjacent if and only if they are in distinct partition classes.) Let $\theta_0$ be the graph on $7$ vertices given in \Cref{fig_theta_0}. This graph is an exception to Wilson's theorem \cite{wilson1974}, which we include as the first case of our main theorem for completeness. 
\begin{figure}[ht]
    \centering
    \begin{tikzpicture}
        \tikzstyle{vertex} = [circle, draw]
        \tikzstyle{edge} = [-]

        \node[vertex](a) at (-2, 0) {};
        \node[vertex](b) at (-1, 1.7) {};
        \node[vertex](c) at (1, 1.7) {};
        \node[vertex](d) at (2, 0) {};
        \node[vertex](e) at (1, -1.7) {};
        \node[vertex](f) at (-1, -1.7) {};
        \node[vertex](g) at (0, 0) {};
        
        \draw[edge] (a)--(b)--(c)--(d)--(e)--(f)--(a);
        \draw[edge] (a)--(g)--(d);
    \end{tikzpicture}
    \caption{The $\theta_0$ graph.}
    \label{fig_theta_0}
\end{figure}
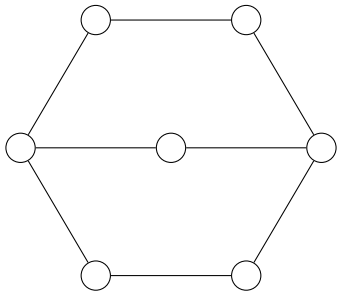
\begin{theorem}\label{thm_main}
    Suppose $t \geq 2$ and $1 \leq k_1 \leq \cdots \leq k_t$, where $n = k_1 + \cdots + k_t \geq 4$. Let $X$ be a graph on $n$ vertices. The following hold:
    \begin{enumerate}
        \item Suppose $k_t = n - 1$. Then $\FS(X, K_{1, n-1})$ is connected if and only if $X$ is connected, is non-bipartite, is not a cycle, is not the graph $\theta_0$ (see \Cref{fig_theta_0}), and does not contain a cut-vertex ($1$-bridge). 
        \item Suppose $t = 2$ and $k_1 \geq 2$. Then $\FS(X, K_{k_1, n - k_1})$ is connected if and only if $X$ is connected, is non-bipartite, is not a cycle, and does not contain a $k_1$-bridge. 
        \item Suppose $t > 2$ and either $k_t > 2$ or $\gcd(k_1, \ldots, k_t) > 1$. Then $\FS(X, K_{k_1, \ldots, k_t})$ is connected if and only if $X$ is connected, is not a cycle, and does not contain an $(n - k_t)$-bridge.
        \item Suppose $t > 2$, $k_1 = 1$, and $k_t = 2$. Then $\FS(X, K_{k_1, \ldots, k_t})$ is connected if and only if $X$ is connected and is not a path (i.e., does not contain an $(n - 2)$-bridge).
    \end{enumerate}
\end{theorem}
Case (1) of \Cref{thm_main} is proved by Wilson in \cite{wilson1974}. Case (2) is Wang, Lu, and Chen's conjecture. We break down cases (2), (3), and (4) into three theorems. The following takes care of the ``only if'' directions of all three cases. 
\begin{theorem}\label{thm_multi_disconnected}
    Suppose $t \geq 2$ and $1 \leq k_1 \leq \cdots \leq k_t$, where $n = k_1 + \cdots + k_t \geq 4$. Let $X$ be a graph on $n$ vertices. The following hold:
    \begin{enumerate}
        \item Suppose $t = 2$ and $k_1 \geq 2$. Then $\FS(X, K_{k_1, n - k_1})$ is disconnected if $X$ is disconnected, is bipartite, is a cycle, or contains a $k$-bridge.
        \item Suppose $t > 2$ and either $k_t > 2$ or $\gcd(k_1, \ldots, k_t) > 1$. Then $\FS(X, K_{k_1, \ldots, k_t})$ is disconnected if $X$ is disconnected, is a cycle, or contains an $(n - k_t)$-bridge.
        \item Suppose $t > 2$, $k_1 = 1$, and $k_t = 2$. Then $\FS(X, K_{k_1, \ldots, k_t})$ is disconnected if $X$ is disconnected or is a path.
    \end{enumerate}
\end{theorem}
The first case above is proved by Wang, Lu, and Chen in \cite{wang2023}. The next theorem takes care of the ``if'' direction for case (2) of \Cref{thm_main} and the ``if'' directions for cases (3) and (4) with the additional assumption that $X$ is not a tree. 
\begin{theorem}\label{thm_multi_connected}
    Suppose $t \geq 2$ and $1 \leq k_1 \leq \cdots \leq k_t$, where $n = k_1 + \cdots + k_t \geq 4$. Let $X$ be a graph on $n$ vertices. The following hold:
    \begin{enumerate}
        \item Suppose $t = 2$ and $k_1 \geq 2$. Then $\FS(X, K_{k_1, n - k_1})$ is connected if $X$ is connected, is non-bipartite, is not a cycle, and does not contain a $k_1$-bridge.
        \item Suppose $t > 2$ and either $k_t > 2$ or $\gcd(k_1, \ldots, k_t) > 1$. Then $\FS(X, K_{k_1, \ldots, k_t})$ is connected if $X$ is connected, is not a cycle, is not a tree, and does not contain an $(n - k_t)$-bridge.
        \item Suppose $t > 2$, $k_1 = 1$, and $k_t = 2$. Then $\FS(X, K_{k_1, \ldots, k_t})$ is connected if $X$ is connected, is not a cycle, and is not a tree. 
    \end{enumerate} 
\end{theorem}
The final theorem takes care of the ``if'' direction for cases (3) and (4) assuming $X$ is a tree.
\begin{theorem}\label{thm_tree}
    Suppose $t > 2$ and $1 \leq k_1 \leq \cdots \leq k_t$, where $n = k_1 + \cdots + k_t \geq 4$. Let $X$ be a graph on $n$ vertices. Then $\FS(X, K_{k_1, \ldots, k_t})$ is connected if $X$ is a tree and does not contain an $(n - k_t)$-bridge.
\end{theorem}
The only missing case is when $X$ is a cycle in case (4) of \Cref{thm_main}. This is taken care of by \Cref{lem_cycle_components}. 

By case (2) of \Cref{thm_main}, $\FS(X, K_{k, n-k})$ has at least two connected components if $X$ is connected and bipartite. It is thus a natural question to consider when there are exactly two connected components.
\begin{theorem}\label{thm_two_comps}
    Suppose $n \geq 5$ and $n \geq 2k \geq 4$. Let $X$ be a connected bipartite graph on $n$ vertices that is not a cycle. Then $\FS(X, K_{k, n-k})$ has exactly two connected components if and only if $X$ does not contain a $k$-bridge and is not the graph $T_6$, $T_7$, or $T_8$ (see \Cref{fig_exceptions}).
    \begin{figure}[ht]
    \centering
    \begin{tikzpicture}
        \tikzstyle{vertex} = [circle, draw]
        \tikzstyle{edge} = [-]

        \node[vertex](1) at (-2, 0) {};
        \node[vertex](2) at (-1, 0) {};
        \node[vertex](3) at (0, 0) {};
        \node[vertex](4) at (1, 0) {};
        \node[vertex](5) at (0, -2) {};
        \node[vertex](6) at (0, -1) {};
        
        \draw[edge] (1)--(2)--(3)--(4);
        \draw[edge] (3)--(6)--(5);

        \node[vertex](1) at (-2+4.5, 0) {};
        \node[vertex](2) at (-1+4.5, 0) {};
        \node[vertex](3) at (0+4.5, 0) {};
        \node[vertex](4) at (1+4.5, 0) {};
        \node[vertex](5) at (2+4.5, 0) {};
        \node[vertex](6) at (0+4.5, -1) {};
        \node[vertex](7) at (0+4.5, -2) {};
        
        \draw[edge] (1)--(2)--(3)--(4)--(5);
        \draw[edge] (3)--(6)--(7);

        \node[vertex](1) at (-2+10, 0) {};
        \node[vertex](2) at (-1+10, 0) {};
        \node[vertex](3) at (0+10, 0) {};
        \node[vertex](4) at (1+10, 0) {};
        \node[vertex](5) at (2+10, 0) {};
        \node[vertex](6) at (3+10, 0) {};
        \node[vertex](7) at (0+10, -1) {};
        \node[vertex](8) at (0+10, -2) {};
        
        \draw[edge] (1)--(2)--(3)--(4)--(5)--(6);
        \draw[edge] (3)--(7)--(8);
    \end{tikzpicture}
    \caption{The three exceptions $T_6$, $T_7$, and $T_8$}
    \label{fig_exceptions}
\end{figure}
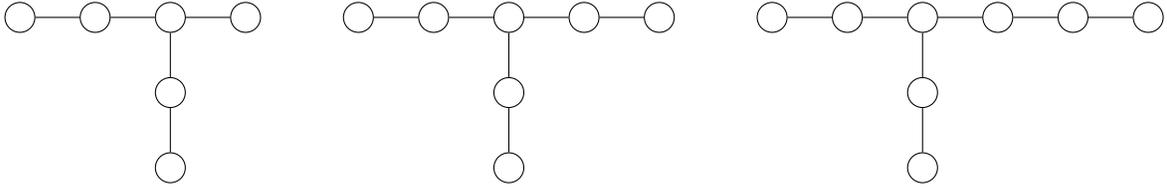
\end{theorem}
The result for $k = 2$ was conjectured by Wang, Lu, and Chen in \cite{wang2023}.

Seemingly unaware of the work of Wang, Lu, and Chen in \cite{wang2023}, Brunck and Kwan solved a similar problem in \cite{brunck2023}. Instead of the complete bipartite graph $K_{k, n - k}$, they consider the book graph $B_{k, n - k}$, which can be obtained from $K_{k, n - k}$ by adding all the edges between pairs of vertices in the size-$k$ class of the vertex bipartition. Observe that $B_{k, n - k}$ is exactly the complete $(k+1)$-partite graph $K_{1, \ldots, 1, n-k}$. 

\begin{theorem}[\cite{brunck2023}]\label{thm_book}
    Let $X$ be a graph on $n \geq 4$ vertices. Define $\kappa(X)$ as follows:
    \begin{enumerate}
        \item if $X$ is disconnected, $\kappa(X) = \infty$;
        \item if $X$ is a cycle, $\kappa(X) = n - 2$;
        \item if $X$ is bipartite, not a cycle, and does not have a cut-vertex, $\kappa(X) = 2$;
        \item if $X$ is the $\theta_0$ graph (see \Cref{fig_theta_0}), $\kappa(X) = 2$;
        \item otherwise, $\kappa(X)$ is the smallest integer $k$ such that $X$ contains no $k$-bridge. 
    \end{enumerate}
    Then the smallest integer $k$ such that $\FS(X, B_{k, n - k})$ is connected is exactly $\kappa(X)$.
\end{theorem}
This follows easily from \Cref{thm_main}, \Cref{lem_cycle_components}, and \Cref{lem_k_bridge_disconnected}.

The rest of the paper is organized as follows. In \Cref{sec_prelim}, we recall some preliminary results and prove \Cref{thm_multi_disconnected}. In \Cref{sec_not_tree}, we prove \Cref{thm_multi_connected} assuming a fact, which we prove in \Cref{sec_stopwatch}. In \Cref{sec_tree}, we prove \Cref{thm_tree}. In \Cref{sec_two_comps}, we prove \Cref{thm_two_comps}. In \Cref{sec_conclusion}, we discuss potential problems for future research. 

\subsection*{Acknowledgements}
This work was done at the University of Minnesota Duluth with support from Jane Street Capital, the National Security Agency, and the CYAN Undergraduate Mathematics Fund at MIT. The author would like to thank Joe Gallian and Colin Defant for organizing the Duluth REU and providing this great research opportunity. The author also thanks Noah Kravitz and Ryan Jeong for providing helpful comments that improved this paper. 

\section{Preliminaries}\label{sec_prelim}
First, we recall some basic properties and definitions of friends-and-strangers graphs. 
\begin{lemma}[\cite{defant2021}]
    For $X, Y$ on $n$ vertices, $\FS(X, Y)$ is isomorphic to $\FS(Y,X)$. 
\end{lemma}
\begin{lemma}[\cite{godsil2001}]\label{lem_complete}
    Let $K_n$ be the complete graph on $n$ vertices and $Y$ a graph on $n$ vertices. Then $\FS(K_n, Y)$ is connected if and only if $Y$ is connected. 
\end{lemma}
\begin{lemma}[\cite{godsil2001}, \cite{defant2021}]\label{lem_path}
    Let $P_n$ be the path on $n$ vertices and $Y$ a graph on $n$ vertices. Then $\FS(P_n, Y)$ is connected if and only if $Y$ is the complete graph. 
\end{lemma}
\begin{lemma}[\cite{defant2021}]\label{lem_spanning_subgraph}
    Let $X, \widetilde{X}, Y, \widetilde{Y}$ be graphs on $n$ vertices. If $X$ and $Y$ are subgraphs of $\widetilde{X}$ and $\widetilde{Y}$, respectively, then $\FS(X, Y)$ is a subgraph of $\FS(\widetilde{X}, \widetilde{Y})$. 
\end{lemma}
\begin{definition}[\cite{alon2023}]\label{def_exchangeable}
    Suppose $\sigma: V(X) \to V(Y)$ is a bijection between vertex sets of two graphs of the same size and $u, v \in V(Y)$. We say that $u$ and $v$ are \blue{\emph{$(X, Y)$-exchangeable from $\sigma$}} if there is a sequence of $(X, Y)$-friendly swaps from $\sigma$ that exchanges $u$ and $v$; namely, it takes $\sigma$ to the bijection $(u \; v) \circ \sigma$.
\end{definition}
The following result relates the connectedness of different friends-and-strangers graphs using exchangeable pairs of vertices. 
\begin{proposition}[\cite{alon2023}]\label{prop_exchange}
    Let $X$, $Y$, and $\widetilde{Y}$ be $n$-vertex graphs such that $Y$ is a spanning subgraph of $\widetilde{Y}$. Suppose that for every edge $\{u, v\}$ of $\widetilde{Y}$ and every bijection $\sigma$ satisfying $\{\sigma^{-1}(u), \sigma^{-1}(v)\} \in E(X)$, the vertices $u$ and $v$ are $(X, Y)$-exchangeable from $\sigma$. Then the connected components of $\FS(X,Y)$ and the connected components of $\FS(X, \widetilde{Y})$ have the same vertex sets. In particular, $\FS(X,Y)$ is connected if and only if $\FS(X, \widetilde{Y})$ is connected. 
\end{proposition}
\begin{proof}
    By \Cref{lem_spanning_subgraph}, it suffices to show that if $\{\sigma,\sigma'\}$ is an edge in $\FS(X, \widetilde{Y})$, then $\sigma$ and $\sigma'$ are in the same connected component in $\FS(X,Y)$. This is precisely the assumption in the statement of the proposition. 
\end{proof}
The following lemmas help us prove \Cref{thm_multi_disconnected}. 

\begin{lemma}[\cite{defant2021}]\label{lem_bipartite_disconnected}
    If $X$ and $Y$ are bipartite graphs on $n$ vertices, then $\FS(X, Y)$ is disconnected.
\end{lemma}

%\begin{lemma}[\cite{defant2021}]\label{lem_cycle_components}
    %Let $Y$ be a graph on $n \geq 3$ vertices and $C_n$ the cycle on $n$ vertices. Then $\FS(C_n, Y)$ is connected if and only if $\overline{Y}$ is a forest consisting of trees $T_1, \ldots, T_r$ such that $\gcd(|V(T_1)|, \ldots, |V(T_r)|) = 1$. 
%\end{lemma}

\begin{lemma}\label{lem_cycle_components}
    Suppose $t \geq 2$, $1 \leq k_1 \leq \cdots \leq k_t$, and $n = k_1 + \cdots + k_t \geq 4$. Let $C_n$ be the cycle on $n$ vertices. Then the number of connected components in $\FS(C_n, K_{k_1, \ldots, k_t})$ is exactly $\gcd(k_1, \ldots, k_t) \cdot \prod (k_i - 1)!$. In particular, $\FS(C_n, K_{k_1, \ldots, k_t})$ is connected if and only if $t > 2$, $k_1 = 1$, and $k_t \leq 2$. 
\end{lemma}
\begin{proof}
    Suppose $C_n$ has vertices $c_1, \ldots, c_n$ (ordered clockwise). We can map any bijection $\sigma$ to an element in the direct product of symmetric groups $\prod_{i = 1}^t S_{k_i}$ by looking at the ordering of vertices in each partition class in $(\sigma(c_1), \ldots, \sigma(c_n))$. Observe that elements in different partition classes can be swapped and a swap along any edge in the cycle (except for the swap along $\{c_1, c_n\}$) of two such elements does not change the element of $\prod_{i = 1}^t S_{k_i}$ the bijection maps to. Thus, we can order the elements of $[n]$ so that for any $\tau \in S_t$, we can move the elements in the partition class of size $k_{\tau(i)}$ to occupy $\{c_{k_{\tau(1)} + \cdots + k_{\tau(i-1)} + 1}, \ldots, c_{k_{\tau(1)} + \cdots + k_{\tau(i)}}\}$, all without changing the element of $\prod_{i = 1}^t S_{k_i}$ that the bijection maps to. 

    The only swap changing the permutation that $\sigma$ maps to is the swap along $\{c_1, c_n\}$. This corresponds precisely to the permutation with the cycle $(1 \; \cdots \; k_{\tau(1)})$ on $S_{k_{\tau(1)}}$ and the cycle $(1 \; \cdots \; k_{\tau(t)})^{-1}$ on $S_{k_{\tau(t)}}$. Since we can choose $\tau \in S_t$ arbitrarily, the set of generators we get from these friendly swaps is 
    \[
        \left\{\left(e, \ldots, e, (1 \; \cdots \; k_{i}), e, \ldots, e, (1 \; \cdots \; k_{j})^{-1}, e, \ldots, e \right) \in \prod_{i = 1}^t S_{k_i} \ \Big{|} \ i, j \in [t] \right\}.
    \]
    Thus, the subgroup of $\prod_{i = 1}^t S_{k_i}$ we can generate is precisely 
    \[
        H = \left\{ \left( (1 \; \cdots \; k_{1})^{a_1}, \ldots, (1 \; \cdots \; k_{t})^{a_t} \right) \ \Big{|} \ \sum_{i = 1}^{t} a_i = 0 \right\}.
    \]
    By B\'ezout's Lemma, $|H| = \prod_{i = 1}^{t} k_i / \gcd(k_1, \ldots, k_t)$. The number of connected components in $\FS(C_n, K_{k_1, \ldots, k_t})$ is then the number of left cosets of $H$ in $\prod_{i = 1}^t S_{k_i}$, which is $\gcd(k_1, \ldots, k_t) \cdot \prod (k_i - 1)!$.
\end{proof}
This can also be deduced from Defant and Kravitz's result on cycle graphs in \cite{defant2021}.

\begin{lemma}[\cite{milojevic2022}]\label{lem_k_bridge_disconnected}
    Let $X$ be a graph containing a $k$-bridge. If $Y$ is not $(k + 1)$-connected, then $\FS(X, Y)$ is disconnected. 
\end{lemma}

Now we are ready to prove the disconnectedness results.
\begin{proof}[Proof of \Cref{thm_multi_disconnected}]
    The first case follows from \Cref{lem_complete}, \Cref{lem_bipartite_disconnected}, \Cref{lem_cycle_components}, and \Cref{lem_k_bridge_disconnected}. The second case follows from \Cref{lem_complete}, \Cref{lem_cycle_components}, and \Cref{lem_k_bridge_disconnected}. The third case follows from \Cref{lem_complete} and \Cref{lem_path}.
\end{proof}

\section{Connectedness of $\FS(X, K_{k_1, \ldots, k_t})$ where $X$ is not a tree}\label{sec_not_tree}
Throughout this section, we assume that $K_{k_1, \ldots, k_t}$ has vertex set $[n]$, with partition classes $S_1 = \{s_1 + 1, \ldots, s_1 + k_1\}, S_2 = \{s_2 + 1, \ldots, s_2 + k_2\}, \ldots, S_t = \{s_{t} + 1, \ldots, s_{t} + k_t\}$, where $s_1 = 0$ and $s_i = \sum_{j = 1}^{i - 1} k_j$. Since Wilson's theorem in \cite{wilson1974} takes care of the case where $K_{k_1, \ldots, k_t}$ is a star, we can assume that $k_t < n - 1$. 

By \Cref{prop_exchange} and \Cref{lem_complete}, the following implies \Cref{thm_multi_connected}. 

\begin{theorem}\label{thm_reduction}
    Suppose $t \geq 2$ and $1 \leq k_1 \leq \cdots \leq k_t$, where $n = k_1 + \cdots + k_t \geq 4$. Let $X$ be a graph on $n$ vertices. Suppose one of the following holds:
    \begin{enumerate}
        \item $t = 2$ and $k_t < n - 1$. The graph $X$ is connected, is non-bipartite, is not a cycle, and does not contain a $k$-bridge.
        \item $t > 2$. The graph $X$ is connected, is not a cycle, is not a tree, and does not contain an $(n - k_t)$-bridge.
    \end{enumerate} 
    Then for any $i, j \in [n]$ and $\sigma: V(X) \to V(K_{k_1, \ldots, k_t})$ satisfying $\{\sigma^{-1}(i), \sigma^{-1}(j)\} \in E(X)$, the vertices $i$ and $j$ are $(X, K_{k_1, \ldots, k_t})$-exchangeable from $\sigma$. 
\end{theorem}
We can merge cases (2) and (3) from \Cref{thm_multi_connected} since in case (3), if $X$ is a cycle, \Cref{lem_cycle_components} shows that $\FS(C_n, K_{k_1, \ldots, k_t})$ is connected. The condition in case (3) that $X$ is not a tree, hence not a path, implies that $X$ does not contain an $(n - 2)$-bridge.

By the symmetry of $K_{k_1, \ldots, k_t}$, it suffices to prove \Cref{thm_reduction} when $(i, j) = (s_a + 1, s_a + 2)$ for any $a \in [t]$ with $k_a > 1$. (The case where $i$ and $j$ are in different partition classes is trivially true because $i$ and $j$ are adjacent in $K_{k_1, \ldots, k_t}$). 

Another important use of the symmetry is that if $i,j$ are in the same partition class, then they can swap with the same set of ``people'' in $[n]$. This allows us to simplify the problem by moving to a more favorable configuration for the exchange. 
\begin{lemma}\label{lem_u_v_equivalent}
    Let $X$ and $Y$ be graphs on $n$ vertices. If $u, v \in V(Y)$ have the same set of neighbors in $Y$, then for any $\sigma$ and $\sigma'$ in the same component of $\FS(X, Y)$, $u$ and $v$ are $(X, Y)$-exchangeable from $\sigma$ if and only if they are $(X, Y)$-exchangeable from $\sigma'$. 
\end{lemma}
\begin{proof}
    Suppose $u$ and $v$ are $(X, Y)$-exchangeable from $\sigma$. Let $S = s_1 s_2 \cdots s_m$ be a sequence of $(X, Y)$-friendly swaps that takes $\sigma$ to $\sigma'$. Namely, each $s_i$ is a transposition on $V(Y)$ and $\sigma' = s_m \circ \cdots \circ s_1 \circ \sigma$. Let $s_i'$ be $s_i$ with any appearance of $u$ replaced with $v$ and vice versa. Because $u$ and $v$ have the same neighbors in $Y$, if $v$ were standing where $u$ is on the graph $X$, $v$ would be able to perform any friendly swap $u$ can perform. Then we claim that 
    \[
    (u \; v) \circ \sigma' = s_m \circ \cdots \circ s_1 \circ (u \; v) \circ \sigma.
    \]
    Indeed, if $s_1 = (a \; u)$, then 
    \[
    s_1 \circ (u \; v) = (u \; v) \circ (a \; v) = (u \; v) \circ s_1'.
    \]
    (The same holds for $s_1 = (a \; v)$.) The claim then follows inductively. Thus, $(u \; v) \circ \sigma'$ and $(u \; v) \circ \sigma$ are also in the same component of $\FS(X, Y)$. Thus, $u$ and $v$ are $(X, Y)$-exchangeable from $\sigma$ if and only if they are $(X, Y)$-exchangeable from $\sigma'$.
\end{proof}
The next straightforward result allows us to localize our problem.
\begin{lemma}\label{lem_localization}
    Let $X$ and $Y$ be graphs on $n$ vertices and $\sigma: V(X) \to V(Y)$ a bijection. Let $X'$ be a subgraph of $X$ with vertex set $A$ and let $Y'$ be the induced subgraph of $Y$ with vertex set $\sigma(A)$. Suppose $\sigma': V(X) \to V(Y)$ agrees with $\sigma$ on $V(X) \setminus A$. If $\sigma|_A$ and $\sigma'|_A$ are in the same component of $\FS(X', Y')$, then $\sigma$ and $\sigma'$ are in the same component of $\FS(X, Y)$.
\end{lemma}
\begin{proof}
    The sequence of friendly swaps taking $\sigma|_A$ to $\sigma'|_A$ in $\FS(X', Y')$ extends to a sequence of friendly swaps taking $\sigma$ to $\sigma'$ in $\FS(X, Y)$.
\end{proof}

For the rest of the section, assume that $X$ is a connected graph on $n$ vertices satisfying the conditions in \Cref{thm_reduction} given $t$ and $k_1, \ldots, k_t$. Since $X$ is not a tree in any of the three cases, it contains a cycle. Suppose that $C$ is a minimal cycle in $X$. Since $X$ is not a cycle, there is at least one edge of $X$ outside $C$. If $V(C) = V(X)$, this edge would split $C$ into two smaller cycles, contradicting the minimality of $C$. Thus, there is an edge connecting $C$ to a vertex $x$ outside of $V(C)$. Let $W$ denote this ``stopwatch'' subgraph (as Brunck and Kwan call it in \cite{brunck2023}) consisting of the cycle and the extra edge. Furthermore, if $t = 2$, we can choose $C$ above to be a minimal odd cycle since $X$ is non-bipartite. 
 
\begin{theorem}\label{thm_stopwatch}
    Let $W$ be a graph on $n \geq 4$ vertices consisting of an $(n-1)$-cycle and a single edge from the cycle to the remaining vertex. Then for any $t \geq 2$ and $1 \leq k_1 \leq \cdots \leq k_t < n - 1$, $\FS(W, K_{k_1, \ldots, k_t})$ is connected if and only if $n$ is even or $t > 2$.
\end{theorem}
We defer the proof of \Cref{thm_stopwatch} to the next section. We also show that we can populate a cycle with any elements of $S_a$ (the partition class containing $s_a + 1$ and $s_a + 2$). 

We first introduce some terminology that we will use for the remainder of the paper. We localize our problem to some subgraph $C \subset X$ (in this current discussion, a minimal cycle in $X$) and would like to show that we can populate it with elements of $[n]$ in a particular way, paying special attention to elements of a certain partition class ($S_a$ in this case). To make our arguments easier to follow and more intuitive, we imagine that there is a fire in a house shaped like $C$. We treat the elements of $S_a$ as \blue{\emph{people}} that stand on the graph $X$, and the elements of $[n] \setminus S_a$ as \blue{\emph{empty spots}}. People can walk into empty spots, but cannot swap with each other. Empty spots of different partition classes can also swap with each other. A subset $S$ of the people are \blue{\emph{firemen}}, and we would like to get these firemen into $C$ and evacuate all other people. The set of firemen will be specified in each context.

We say that there is a \blue{\emph{fire exit}} from $u \in V(X)$ to $v \in V(X)$ if there is a path from $u$ to $v$ and $v$ is currently occupied by an empty spot. Whenever there is such a fire exit, we may move everyone down the path so that $u$ becomes the empty spot. We say that two fire exits have \blue{\emph{different exits}} if they lead to different vertices in $X$, and that they do not \blue{\emph{block}} each other if they have different exits and neither path contains the other. 
\begin{proposition}\label{prop_populate_cycle}
    Suppose $t \geq 2$ and $1 \leq k_1 \leq \cdots \leq k_t$, where $n = k_1 + \cdots + k_t \geq 4$. Let $X$ be a connected graph on $n$ vertices that is not a cycle, is not a tree, and does not contain an $(n - k_{t})$-bridge. Let $C$ be a cycle in $X$. For any subset of any partition class $S \subset S_a$ with $|V(C)| + |S_a| - n \leq |S| < |V(C)|$ and a bijection $\sigma: V(X) \to V(K_{k_1, \ldots, k_t})$, there exists a bijection $\sigma'$ in the same connected component of $\FS(X, K_{k_1, \ldots, k_t})$ as $\sigma$ such that $\sigma'(V(C)) \cap S_a = S$. 
\end{proposition}
\begin{proof}
    \textbf{Case I: $|V(C)| \geq |S_a| + 1$.} In this case, we can simply move all the people into $C$, and we will still have at least one empty spot in $C$. This can be done very easily because $X$ is connected. Now we just need to evacuate the people who are not firemen. Since there are $n - k_a \geq |V(C)| - |S|$ empty spots, there are enough empty spots to replace the evacuated people in $C$. Suppose $p_1$ is a person that needs to be evacuated. Since there must be an empty spot outside $C$, we can move the people in $C$ around so that $p_1$ is closest to the fire exit to the empty spot outside $C$ (since $X$ is connected, there exists such a fire exit). Then $p_1$ can just evacuate through this fire exit (see \Cref{fig_2}). 

    \begin{figure}[ht]
        \centering
        \begin{tikzpicture}[node distance = 1 cm, auto, scale = 0.9]
            \tikzstyle{spot} = [circle, draw, inner sep = 0pt]
            \tikzstyle{edge} = [-]
            \tikzstyle{person} = [circle, draw, fill = red!50, inner sep = 0pt]
        
            \node[person](a) at (-2, 0) {\phantom{$f_1$}};
            \node[person](b) at (-1, 1.7) {\phantom{$f_1$}};
            \node[person](c) at (1, 1.7) {\phantom{$f_1$}};
            \node[spot](d) at (2, 0) {\phantom{$f_1$}};
            \node[person](e) at (1, -1.7) {\phantom{$f_1$}};
            \node[person](f) at (-1, -1.7) {\phantom{$f_1$}};
            \node[spot](g) at (4, 3) {\phantom{$f_1$}};

            \node () at (b) {$p_1$};
            
            \draw[edge] (a)--(b)--(c)--(d)--(e)--(f)--(a);
            \path (c) edge [dotted, very thick] (g);
            \path (b) edge [bend left, dotted, very thick, ->, color = blue!50] node{$2$} (g);
            \path (c) edge [bend left, dotted, very thick, ->, color = blue!50] node{$1$} (d);
        \end{tikzpicture}
        \caption{People are nodes filled in red and empty spots white. A dotted line is a fire exit. Blue arrows indicate movement of people. Numbers on blue arrows indicate the order of the movement. In this diagram, $p_1$ can evacuate through the fire exit. (Of course, $p_1$ may not reach the empty spot. Rather, $p_1$ will push everyone down the fire exit by one step.)}
        \label{fig_2}
    \end{figure}
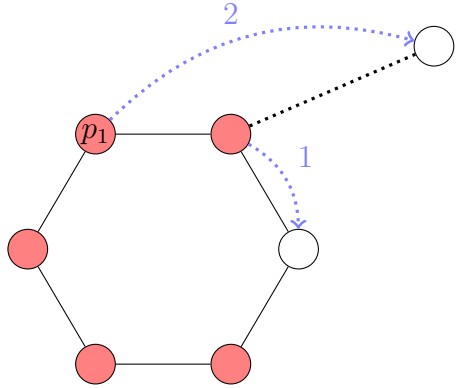

    \textbf{Case II: $|V(C)| \leq |S_a|$.} Now the house on fire does not have enough space for all the people. Our goal is to maintain an empty spot in $C$ so that the people in $C$ can move around. We can always do this as our first step because $X$ is connected. (This could come at the cost of having to remove a fireman from $C$, but we can bring them back later.) So now suppose there is an empty spot in $C$. If all firemen are already in $C$, we can just evacuate the remaining non-firemen people in $C$ as in Case I. Now suppose $f_1$ is a fireman not in $C$. Suppose $x_1 \in V(C), x_2, \ldots, x_\ell= \sigma^{-1}(f_1)$ is a shortest path from $C$ to where $f_1$ is (see \Cref{fig_3}). Let $Q = V(C) \cup \{x_2, \ldots, x_{\ell - 1}\}$ be the set of vertices ``in the way'' of $f_1$. We induct on $\ell$ and show that we can always reduce $\ell$ by $1$ while keeping the set of firemen in $C$ intact and still having an empty spot in $C$. 

    \begin{figure}[ht]
        \centering
        \begin{tikzpicture}[node distance = 2 cm, auto, scale = 0.9]
        \tikzstyle{spot} = [circle, draw, inner sep = 0pt]
        \tikzstyle{edge} = [-]
        \tikzstyle{person} = [circle, draw, fill = red!50, inner sep = 0pt]
            
        \node[person](a) at (-2, 0) {\phantom{$f_1$}};
        \node[spot](b) at (-1, 1.7) {\phantom{$f_1$}};
        \node[spot](c) at (1, 1.7) {\phantom{$f_1$}};
        \node[person](d) at (2, 0) {\phantom{$f_1$}};
        \node[person](e) at (1, -1.7) {\phantom{$f_1$}};
        \node[person](f) at (-1, -1.7) {\phantom{$f_1$}};
        \node[person](g) at (4, 0) {\phantom{$f_1$}};

        \node () at (g) {$f_1$};

        \draw[edge] (a)--(b)--(c)--(d)--(e)--(f)--(a);
        \path (d) edge [dotted, very thick] node{$\ell$} (g);
        \end{tikzpicture}
        \caption{The fireman $f_1$ is currently at distance $\ell - 1$ from $C$. The goal is to bring $f_1$ closer.}
        \label{fig_3}
    \end{figure}
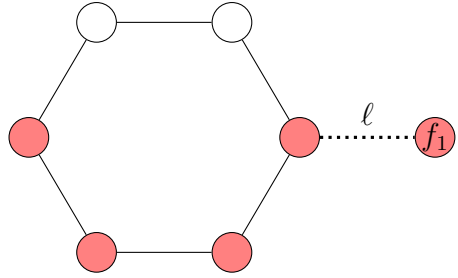
    
    \begin{enumerate}
        \item If there are at least two empty spots in $Q$, then we can simply push $f_1$ in until there is only one empty spot in $Q$ or $f_1$ is in $C$ (see \Cref{fig_4}). Note that we keep one empty spot in $C$ in the process so that people can move around inside $C$.
        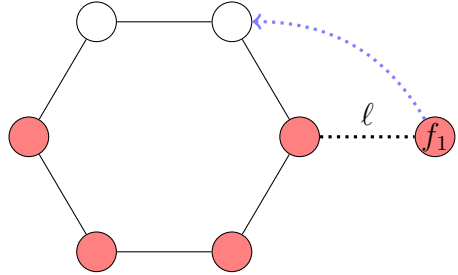
\begin{figure}[ht]
            \centering
            \begin{tikzpicture}[node distance = 2 cm, auto, scale = 0.9]
            \tikzstyle{spot} = [circle, draw, inner sep = 0pt]
            \tikzstyle{edge} = [-]
            \tikzstyle{person} = [circle, draw, fill = red!50, inner sep = 0pt]
                
            \node[person](a) at (-2, 0) {\phantom{$f_1$}};
            \node[spot](b) at (-1, 1.7) {\phantom{$f_1$}};
            \node[spot](c) at (1, 1.7) {\phantom{$f_1$}};
            \node[person](d) at (2, 0) {\phantom{$f_1$}};
            \node[person](e) at (1, -1.7) {\phantom{$f_1$}};
            \node[person](f) at (-1, -1.7) {\phantom{$f_1$}};
            \node[person](g) at (4, 0) {\phantom{$f_1$}};

            \node () at (g) {$f_1$};
                        
            \draw[edge] (a)--(b)--(c)--(d)--(e)--(f)--(a);
            \path (d) edge [dotted, very thick] node{$\ell$} (g);
            \path (g) edge [bend right, dotted, very thick, ->, color = blue!50] (c);
            \end{tikzpicture}
            \caption{If there are at least two empty spots, we can directly move $f_1$ closer to $C$.}
            \label{fig_4}
        \end{figure}

        \item Otherwise, if there is a fire exit from some vertex in $Q$ that does not go through $f_1$, we can evacuate a non-fireman through that fire exit to get case (1) (see \Cref{fig_5}). 
        \begin{figure}[ht]
            \centering
            \begin{tikzpicture}[node distance = 2 cm, auto, scale = 0.9]
            \tikzstyle{spot} = [circle, draw, inner sep = 0pt]
            \tikzstyle{edge} = [-]
            \tikzstyle{person} = [circle, draw, fill = red!50, inner sep = 0pt]
                
            \node[person](a) at (-2, 0) {\phantom{$f_1$}};
            \node[person](b) at (-1, 1.7) {\phantom{$f_1$}};
            \node[spot](c) at (1, 1.7) {\phantom{$f_1$}};
            \node[person](d) at (2, 0) {\phantom{$f_1$}};
            \node[person](e) at (1, -1.7) {\phantom{$f_1$}};
            \node[person](f) at (-1, -1.7) {\phantom{$f_1$}};
            \node[person](g) at (4, 0) {\phantom{$f_1$}};
            \node[spot](h) at (4, 3) {\phantom{$f_1$}};

            \node () at (g) {$f_1$};
            \node () at (b) {$p_1$};
                        
            \draw[edge] (a)--(b)--(c)--(d)--(e)--(f)--(a);
            \path (d) edge [dotted, very thick] node{$\ell$} (g);
            \path (c) edge [dotted, very thick] (h);
            \path (b) edge [bend left, dotted, very thick, ->, color = blue!50] (h);
            \end{tikzpicture}
            \caption{If some fire exit does not go through $f_1$, then we can evacuate a non-fireman person to get case (1).}
            \label{fig_5}
        \end{figure}
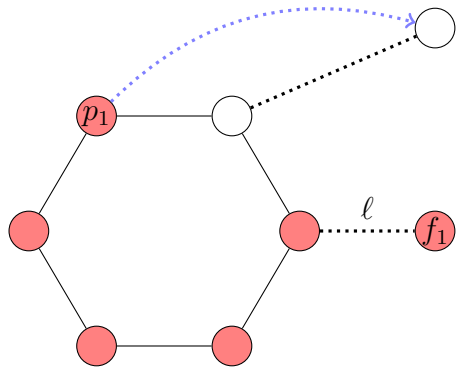

        \item Otherwise, all fire exits from $Q$ must go through $f_1$. This means that the vertices $x_1, x_2, \ldots, x_\ell$ are all cut-vertices. Indeed, if we removed any of them, there would be no path from a vertex in $C$ to an empty spot that is not on $Q$. We first move all empty spots outside $Q$ so that they are as close to $f_1$ as possible. In particular, this means that the shortest path from $f_1$ to every empty spot outside $Q$ has only one person, $f_1$, standing on it. We claim that there are two fire exits that do not block each other from $c_1$. Suppose for the sake of contradiction that this were not the case. Then all the empty spots lie on one path from $c_1$. Since there are at least $n - k_a - 1 \geq n - k_t - 1$ empty spots outside $Q$, these together with $x_1, x_2, \ldots, x_\ell$ would form a bridge of length $n - k_t - 2 +\ell\geq n - k_t$, which violates our assumption. Thus, there must be a ``fork'' in the empty spots and there must be at least two empty spots beyond the fork. We can move $f_1$ to one of the empty spots beyond the fork and evacuate the person at $x_{\ell - 1}$ to the other (see \Cref{fig_6}). Finally, we can move $f_1$ back to $x_{\ell - 1}$. When $\ell = 2$, we might not want to evacuate the person at $x_1$; instead, we can rotate the people around $C$ to evacuate a non-fireman person. 

        \begin{figure}[ht]
            \centering
            \begin{tikzpicture}[node distance = 2 cm, auto, scale = 0.9]
                \tikzstyle{spot} = [circle, draw, inner sep = 0pt]
                \tikzstyle{edge} = [-]
                \tikzstyle{person} = [circle, draw, fill = red!50, inner sep = 0pt]

                \node[person](a) at (-2, 0) {\phantom{$f_1$}};
                \node[person](b) at (-1, 1.7) {\phantom{$f_1$}};
                \node[spot](c) at (1, 1.7) {\phantom{$f_1$}};
                \node[person](d) at (2, 0) {\phantom{$f_1$}};
                \node[person](e) at (1, -1.7) {\phantom{$f_1$}};
                \node[person](f) at (-1, -1.7) {\phantom{$f_1$}};
                \node[person](g) at (5, 0) {\phantom{$f_1$}};
                \node[person](l) at (4, 0) {\phantom{$f_1$}};
                \node[spot](h) at (6, 0) {\phantom{$f_1$}};
                \node[spot](i) at (9, 0) {\phantom{$f_1$}};
                \node[spot](j) at (10, 1) {\phantom{$f_1$}};
                \node[spot](k) at (10, -1) {\phantom{$f_1$}};
                
                \node () at (g) {$f_1$};

                \draw[edge] (a)--(b)--(c)--(d)--(e)--(f)--(a);
                \path (d) edge [dotted, very thick] node{$\ell - 1$} (l);
                \draw[edge] (l)--(g)--(h);
                \draw[edge] (i)--(j);
                \draw[edge] (i)--(k);
                \path (h) edge [dotted, very thick] node{$\leq n - k_t - 3$} (i);
                \path (g) edge [bend left = 15, dotted, very thick, ->, color = blue!50] node{$1$} (j);
                \path (l) edge [bend right, dotted, very thick, ->, color = blue!50] node{$2$} (k);
                \path (j) edge [bend right = 70, dotted, very thick, ->, color = blue!50] node{$3$} (l);
            \end{tikzpicture}
            \caption{The fireman $f_1$ goes to one empty spot beyond the ``fork,'' and $p_1$ evacuates to another empty spot beyond the ``fork.''}
            \label{fig_6}
        \end{figure}
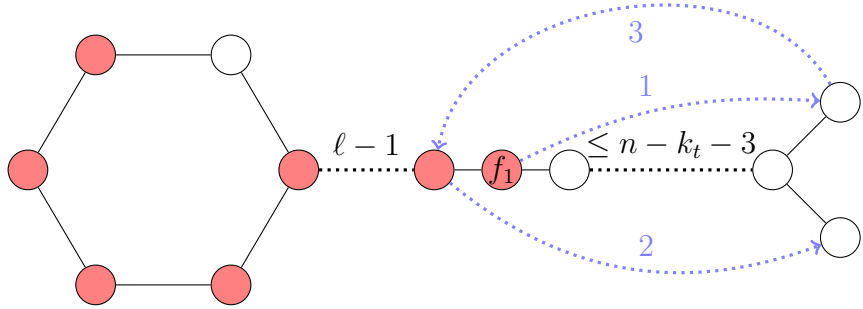
    \end{enumerate}
\end{proof}

With \Cref{prop_populate_cycle}, we are ready to prove the reduction of \Cref{thm_multi_connected}. 
\begin{proof}[Proof of \Cref{thm_reduction} assuming \Cref{thm_stopwatch}]
    Let $C$ be a minimal (odd if $t = 2$) cycle in $X$ with vertices $c_1, \ldots, c_{m}$ in clockwise order. Suppose $s_a + 1, s_a + 2 \in S_a$ are the two vertices of $K_{k_1, \ldots, k_t}$ we want to exchange from some $\sigma$ with $\{\sigma^{-1}(s_a + 1), \sigma^{-1}(s_a + 2)\} \in E(X)$. Let $S \subset S_a$ be any set containing $s_a + 1$ and $s_a + 2$ satisfying $|V(C)| + |S_a| - n \leq |S| < |V(C)|$. By \Cref{prop_populate_cycle}, we can find a sequence of friendly swaps from $\sigma$ to some $\sigma'$ such that $\sigma'(V(C)) \cap S_a = S$. 
    
    If $t = 2$ and there is only one empty spot in $C$, we can move some empty spot outside $C$ to a vertex adjacent to $C$ to form $W$. If $t > 2$ and all empty spots in $C$ are from the same partition class, we can move some empty spot of a different partition class to a vertex adjacent to $C$ to form $W$. Then the induced subgraph of $K_{k_1, \ldots, k_t}$ on $\sigma(V(W))$ is isomorphic to some $K_{k'_1, \ldots, k'_{t'}}$ with $1 \leq k'_1 \leq \cdots \leq k'_{t'} \leq |V(W)| - 1$ and $t' > 2$ whenever $t > 2$. Now \Cref{thm_stopwatch} and \Cref{lem_localization} imply that $s_a + 1$ and $s_a + 2$ are $(X, K_{k_1, \ldots, k_t})$ exchangeable from $\sigma'$. Since $s_a + 1$ and $s_a + 2$ have the same neighbors, \Cref{lem_u_v_equivalent} then implies they are $(X, K_{k_1, \ldots, k_t})$ exchangeable from $\sigma$, which finishes the proof.   
\end{proof}

\section{The proof of \Cref{thm_stopwatch}}\label{sec_stopwatch}
As we will see in this section, \Cref{thm_stopwatch} can be turned into a purely group-theoretic problem. We will make use of the following well-known facts.
\begin{lemma}\label{lem_generators}
    For $n \geq 3$, the symmetric group $S_n$ can be generated by the cycles $(1 \; 2)$ and $(1 \; 2 \cdots n)$.
\end{lemma}
\begin{lemma}\label{lem_normal_subgroup}
    For $n \geq 3, n \neq 4$, the only proper non-trivial normal subgroup of the symmetric group $S_n$ is the alternating group $A_n$. When $n = 4$, there is another proper non-trivial normal subgroup: $\{e, (1 \; 2) (3 \; 4), (1 \; 3) (2 \; 4), (1 \; 4) (2 \; 3)\}$.
\end{lemma}
With these, we show that the direct product of two symmetric groups whose sizes have the same parity can be generated by two specific elements.
\begin{proposition}\label{prop_alpha_beta}
    Suppose $k,\ell\geq 2$ are integers. Let $H$ be the subgroup of $S_k \times S_\ell$ generated by the two elements
    \[
        \alpha = ((2 \; 3 \cdots k), (1 \; 2 \cdots \ell)^{-1}), \quad \text{ and } \quad 
        \beta = ((1 \; 2 \cdots k)^{-1}, (2 \; 3 \cdots \ell)).
    \]
    Then $H = S_k \times S_\ell$ if $k + \ell$ is even, and $H = \{(\pi, \tau)| \operatorname{sgn}(\pi) = \operatorname{sgn}(\tau)\}$ if $k + \ell$ is odd.
\end{proposition}
\begin{proof}
    Let $A = \{(\pi, \tau)| \operatorname{sgn}(\pi) = \operatorname{sgn}(\tau)\}$ be the subgroup of $S_k \times S_\ell$ consisting of even permutations when viewed as a subgroup of $S_{k + l}$. If $k + \ell$ is odd, then $\alpha, \beta \in A$, so $H$ must be a subgroup of $A$.

    \begin{enumerate}
        \item Suppose $k = 2$ and $\ell$ is odd. Since $(1 \; 2 \cdots \ell)^{-1} (2 \; 3 \cdots \ell) = (1 2)$, \Cref{lem_generators} implies that $|H| \geq l! = |A|$. Thus, $H = A$.
        
        \item Suppose $k = 2$ and $\ell$ is even. Then $\beta^{(l-1)} = ((1 \; 2), e)$ generates $S_k \times \{e\}$. Since $(1 \; 2 \cdots \ell)^{-1} (2 \; 3 \cdots \ell) = (1 2)$, \Cref{lem_generators} implies that we can also generate $\{e\} \times S_\ell$. 
        
        \item Suppose $k =\ell= 4$. Then $\alpha^{3} = (e, (1 \; 2 \; 3 \; 4))$ and $\alpha^{4} = ((2 \; 3 \; 4), e)$. Let $\gamma = \beta \alpha^4 = ((1 \; 2), (2 \; 3 \; 4))$. Then $\gamma^{3} = ((1 \; 2), e)$ and $\gamma^{4} = (e, (2 \; 3 \; 4))$. We are then done by \Cref{lem_generators}.

        \item Now suppose $k,\ell\geq 3$ and $k \neq 4$. Let $H$ be the subgroup of $S_k \times S_\ell$ generated by $\alpha$ and $\beta$ and let $H_k = H \cap (S_k \times \{e\})$. By \Cref{lem_generators}, for any $\pi \in S_k$, there exists some $f(\pi) \in S_\ell$ so that $(\pi, f(\pi)) \in H$. For any $(\sigma, e) \in H_k$ and any $(\pi, e) \in S_k \times \{e\}$, 
        \[
            (\pi, e) (\sigma, e) (\pi, e)^{-1} = (\pi, f(\pi)) (\sigma, e) (\pi, f(\pi))^{-1} \in H_k.
        \]
        Thus $H_k$ is a normal subgroup of $(S_k \times \{e\}) \cong S_k$. By \Cref{lem_normal_subgroup}, $H_k \geq A_k \times \{e\}$. Similarly, if $H_\ell= H \cap (\{e\} \times S_\ell)$, then $H_\ell$ is a normal subgroup of $(\{e\} \times S_\ell)$. Since either $(e, (1 \; 2 \cdots \ell)) \in H$ or $(e, (2 \; 3 \cdots \ell)) \in H$, this excludes the other normal subgroup of $S_4$ in the case $\ell = 4$. Therefore, $H_\ell\geq \{e\} \times A_\ell$. It follows that $H \geq A_k \times A_\ell$, but since at least one of $\alpha, \beta$ is not in $A_k \times A_\ell$, $H \geq A$. 
    
        If $k + \ell$ is even, $\alpha \notin A$. Therefore in this case $H = S_k \times S_\ell$.
    \end{enumerate}
\end{proof}
We know prove \Cref{thm_stopwatch}.
\begin{proof}[Proof of \Cref{thm_stopwatch}]
    Suppose the vertices of the stopwatch $W$ are $w_0, w_1, \ldots, w_{n-1}$, where $w_0$ is the vertex not in the cycle, $w_1$ is its neighbor, and $w_1, \ldots, w_{n-1}$ is a clockwise ordering of the vertices in the cycle. If $k_1 = \cdots = k_t = 1$, we are already done by \Cref{lem_complete}. Now suppose $k_t > 1$. Since $k_t < n - 1$, $k_1 + \cdots + k_{t - 1} \geq 2$. 
    
    For any bijection $\sigma$, $(\sigma(w_0), \sigma(w_1), \ldots, \sigma(w_{n - 1}))$ can be viewed as a permutation in $S_{k_t} \times S_{n - k_t}$ by looking at the order of the vertices in the partition class of size $k_t$ separately from the rest. Since the only swap that changes the corresponding permutation is the one across the edge $\{w_1, w_{n - 1}\}$, we can always perform friendly swaps from $\sigma$ to move the $k_t$ vertices to occupy $\{w_0, \ldots, w_{k_t - 1}\}$ or move them to occupy $\{w_{n - k_t + 1}, \ldots, w_{n}\}$ without changing the corresponding permutation. When we are in either of these situations, the swap across $\{w_1, w_{n - 1}\}$ is a valid friendly swap and induces the permutations $\alpha = ((2 \; 3 \cdots k_t), (1 \; 2 \cdots (n-k_t))^{-1})$ and $\beta = ((1 \; 2 \cdots k_t)^{-1}, (2 \; 3 \cdots (n-k_t)))$ respectively (see \Cref{fig_7}). 

    \begin{figure}[ht]
    \centering
        \begin{tikzpicture}[node distance = 2 cm, auto, scale = 0.9]
        \tikzstyle{edge} = [-]
        \tikzstyle{person} = [circle, draw, fill = red!50, inner sep = 2pt]
        \tikzstyle{person2} = [circle, draw, fill = yellow!50, inner sep = 2pt]
        \tikzstyle{person3} = [circle, draw, fill = blue!50, inner sep = 2pt]
            
        \node[person3](a) at (-2, 0) {$2$};
        \node[person3](b) at (-1, 1.7) {$1$};
        \node[person](c) at (1, 1.7) {$3$};
        \node[person](d) at (2, 0) {$2$};
        \node[person2](e) at (1, -1.7) {$4$};
        \node[person2](f) at (-1, -1.7) {$3$};
        \node[person](g) at (4, 0) {$1$};
                    
        \draw[edge] (a)--(b)--(c)--(d)--(e)--(f)--(a);
        \draw[edge] (d)--(g);
        \path (d) edge [bend left, dotted, very thick, <->, color = blue!50] (e);

        \node[person](a') at (6, 0) {$1$};
        \node[person2](b') at (7, 1.7) {$4$};
        \node[person2](c') at (9, 1.7) {$3$};
        \node[person3](d') at (10, 0) {$2$};
        \node[person](e') at (9, -1.7) {$3$};
        \node[person](f') at (7, -1.7) {$2$};
        \node[person3](g') at (12, 0) {$1$};
                    
        \draw[edge] (a')--(b')--(c')--(d')--(e')--(f')--(a');
        \draw[edge] (d')--(g');
        \path (d') edge [bend left, dotted, very thick, <->, color = blue!50] (e');
    \end{tikzpicture}
    \caption{An example where $k_1 = k_2 = 2$ and $k_3 = 3$. The swap across $\{w_1, w_{5}\}$ induces the permutations $\alpha = ((2 \; 3), (1 \; 2 \; 3 \; 4)^{-1})$ and $\beta = ((1 \; 2 \; 3)^{-1}, (2 \; 3 \; 4))$ on $S_{k_3} \times S_{k_1 + k_2}$. }
    \label{fig_7}
    \end{figure}
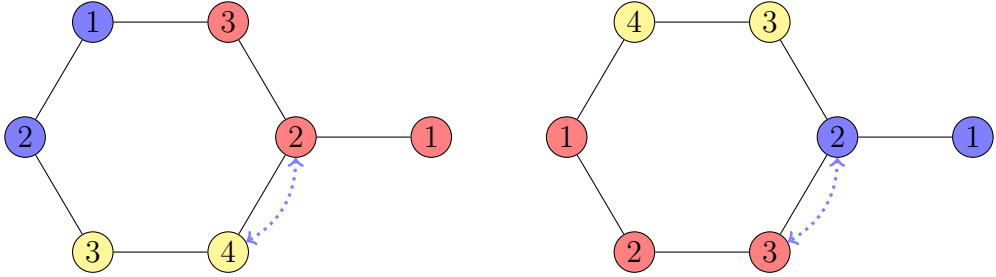

    Thus, \Cref{prop_exchange}, \Cref{lem_u_v_equivalent}, and \Cref{prop_alpha_beta} imply that $\FS(W, K_{k_1, \ldots, k_t})$ has at most two connected components. 

    If $t = 2$ and $n$ is odd, then $\FS(W, K_{k_1, \ldots, k_t})$ has exactly two connected components by \Cref{prop_alpha_beta}.
    
    If $n$ is even, then $\FS(W, K_{k_1, \ldots, k_t})$ is connected by \Cref{prop_alpha_beta}. If $t \geq 3$ and $n$ is odd, any friendly swap within the vertices not in the $k_t$ set produces a single transposition, which connects the two components and thus $\FS(W, K_{k_1, \ldots, k_t})$ is connected. 
\end{proof}

\section{Connectedness of $\FS(X, K_{k_1, \ldots, k_t})$ where $X$ is a tree}\label{sec_tree}
In this section, we solve the final piece of the puzzle: \Cref{thm_tree}. Throughout this section, we assume that $K_{k_1, \ldots, k_t}$ has vertex set $[n]$, with partition classes $S_1 = \{s_1 + 1, \ldots, s_1 + k_1\}, S_2 = \{s_2 + 1, \ldots, s_2 + k_2\}, \ldots, S_t = \{s_{t} + 1, \ldots, s_{t} + k_t\}$, where $s_1 = 0$ and $s_i = \sum_{j = 1}^{i - 1} k_j$.

Similar to how we reduced \Cref{thm_multi_connected} to \Cref{thm_reduction}, we can use \Cref{prop_exchange} to reduce \Cref{thm_tree} to the following.

\begin{theorem}\label{thm_tree_reduction}
    Suppose $t > 2$ and $1 \leq k_1 \leq \cdots \leq k_t$, where $n = k_1 + \cdots + k_t \geq 4$. Let $X$ be a tree on $n$ vertices. If $X$ does not contain an $(n - k_t)$-bridge, then for any $i, j \in [n]$ and $\sigma: V(X) \to V(K_{k_1, \ldots, k_t})$ satisfying $\{\sigma^{-1}(i), \sigma^{-1}(j)\} \in E(X)$, the vertices $i$ and $j$ are $(X, K_{k_1, \ldots, k_t})$-exchangeable from $\sigma$. 
\end{theorem}

Again, it suffices to consider the case $(i, j) = (s_a + 1, s_a + 2)$ for any $a \in [t]$ with $k_a > 1$. Suppose $\sigma^{-1}(s_a + 1) = t_1$ and $\sigma^{-1}(s_a + 1) = t_2$ Since $X$ is now a tree, we cannot localize the problem to a minimal cycle. Instead, we find what Brunck and Kwan called in \cite{brunck2023} a ``snake tongue'': a graph $Y$ consisting of a path $t_1, \ldots, t_\ell$ with $t_\ell$ adjacent to two other vertices $u_1$ and $u_2$. As in the proof of \Cref{thm_reduction}, we show that we can find such a graph in $X$ and correctly populate it with vertices of $K_{k_1, \ldots, k_t}$. We first prove that we can move any two elements in $[n]$ of the same partition class to any two adjacent vertices in $X$.

\begin{lemma}\label{lem_tree_move_anywhere}
    Suppose $t \geq 2$ and $1 \leq k_1 \leq \cdots \leq k_t$, where $n = k_1 + \cdots + k_t \geq 4$. Fix any $i, j \in [n]$ of the same partition class of size $k_a$, $\{u, v\} \in E(X)$, and $\sigma: V(X) \to [n]$. If $X$ is a tree that does not contain an $(n - k_a)$-bridge, then there exists some $\sigma'$ with $\sigma'(\{u, v\}) = \{i, j\}$ that is in the same connected component of $\FS(X, K_{k_1, \ldots, k_t})$ as $\sigma$. 
\end{lemma}
\begin{proof}
    If $X$ is a path, then $K_{k_1, \ldots, k_t}$ is the complete graph and we are done by \Cref{lem_path}. Now suppose $X$ contains a vertex of degree at least $3$. 
    
    We first bring $i$ and $j$ together. Let $(x_1 = \sigma^{-1}(i), x_2, \ldots, x_\ell= \sigma^{-1}(j))$ be the path from where $i$ is to where $j$ is. If $\ell = 2$, we are already done. We now show that $\ell$ can be decreased whenever $\ell > 2$. We treat the elements in the same partition class as $i$ and $j$ as people and others as empty spots.

    If there is an empty spot between $x_1$ and $x_\ell$, we can directly move $i$ and $j$ closer together. Otherwise, as long as there is a fire exit from $x_1$ not passing through $x_2$, we can move all the $\ell$ people on the path down the fire exit. This procedure does not change $\ell$. 
    
    Now we can assume that all fire exits from $x_1$ pass through $x_2$. If there is a fire exit from some $x_a$ not passing through $x_\ell$, then we can move the person at $x_a$ out of the path and reduce $\ell$ by $1$. So suppose that all fire exits from each $x_a$ pass through $x_\ell$. We bring all empty spots so that they are as close to $x_\ell$ as possible. As in the proof of \Cref{prop_populate_cycle}, 
    these empty spots cannot be in one path, since they would form a $(n - k_a)$ bridge with $x_\ell$. Thus, they must have a fork, and we can move $j$ into one vertex beyond the fork and the person at $x_{\ell-1}$ into another vertex beyond the fork (see \Cref{fig_8}). Then, we can move $j$ to $x_{\ell - 1}$ to decrease $\ell$ by $1$. 

    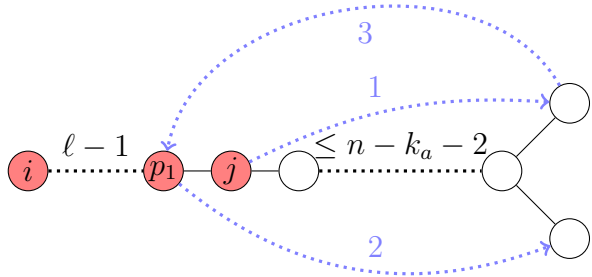
\begin{figure}[ht]
            \centering
            \begin{tikzpicture}[node distance = 2 cm, auto, scale = 0.9]
                \tikzstyle{spot} = [circle, draw, inner sep = 0pt]
                \tikzstyle{edge} = [-]
                \tikzstyle{person} = [circle, draw, fill = red!50, inner sep = 0pt]

                \node[person](d) at (2, 0) {\phantom{$f_1$}};
                \node[person](g) at (5, 0) {\phantom{$f_1$}};
                \node[person](l) at (4, 0) {\phantom{$f_1$}};
                \node[spot](h) at (6, 0) {\phantom{$f_1$}};
                \node[spot](i) at (9, 0) {\phantom{$f_1$}};
                \node[spot](j) at (10, 1) {\phantom{$f_1$}};
                \node[spot](k) at (10, -1) {\phantom{$f_1$}};
                
                \node () at (d) {$i$};
                \node () at (g) {$j$};
                \node () at (l) {$p_1$};

                \path (d) edge [dotted, very thick] node{$\ell - 1$} (l);
                \draw[edge] (l)--(g)--(h);
                \draw[edge] (i)--(j);
                \draw[edge] (i)--(k);
                \path (h) edge [dotted, very thick] node{$\leq n - k_a - 2$} (i);
                \path (g) edge [bend left = 15, dotted, very thick, ->, color = blue!50] node{$1$} (j);
                \path (l) edge [bend right, dotted, very thick, ->, color = blue!50] node{$2$} (k);
                \path (j) edge [bend right = 70, dotted, very thick, ->, color = blue!50] node{$3$} (l);
            \end{tikzpicture}
            \caption{The person $j$ goes to one empty spot beyond the ``fork,'' and $p_1$ evacuates to another empty spot beyond the ``fork.''}
            \label{fig_8}
        \end{figure}

    Now we may assume $i$ and $j$ are adjacent in $X$. Let $y_1 = u, y_2 = v, \ldots, y_{m - 1} = \sigma^{-1}(i), y_{m} = \sigma^{-1}(j)$ be a path from $u$ to $\sigma^{-1}(j)$ (at least one of the four paths from $\{u, v\}$ to $\{\sigma^{-1}(i), \sigma^{-1}(j)\}$ will contain all four vertices). If $m = 2$, we are done. If there is a fire exit from some $y_a$ not passing through $y_{m - 1}$ or $y_{m}$, we can move the person at $y_a$ out of the way and bring $i, j$ closer to $u, v$. If there is an empty spot at some $y_a$ for $a \leq m-2$, then we can also move $i, j$ closer to $u, v$. Now we have the following possibilities. Notice that there are at least three empty spots since there are at least three people if $m > 2$.
    \begin{enumerate}
        \item If there are at least two fire exits that are disjoint except possibly for the vertex they leave from, then since there are at least three empty spots, we can move $i, j$ into fire exits and leave another disjoint fire exit open. Then we can evacuate the person at $y_{m - 2}$. Notice that when $i, j$ go back to $y_{m - 2}, y_{m - 1}$, they might be in a different order, but this is not a problem for us (see \Cref{fig_9}).
        \begin{figure}[ht]
            \centering
            \begin{tikzpicture}[node distance = 2 cm, auto, scale = 0.9]
                \tikzstyle{spot} = [circle, draw, inner sep = 0pt]
                \tikzstyle{edge} = [-]
                \tikzstyle{person} = [circle, draw, fill = red!50, inner sep = 0pt]

                \node[person](c) at (1, 0) {\phantom{$f_1$}};
                \node[person](d) at (3, 0) {\phantom{$f_1$}};
                \node[person](g) at (5, 0) {\phantom{$f_1$}};
                \node[person](l) at (4, 0) {\phantom{$f_1$}};
                \node[spot](a) at (6, 2) {\phantom{$f_1$}};
                \node[spot](b) at (7, 2) {\phantom{$f_1$}};
                \node[spot](e) at (8, 0) {\phantom{$f_1$}};

                \node () at (g) {$j$};
                \node () at (l) {$i$};
                \node () at (d) {$p_1$};

                \draw[edge] (d)--(l);
                \path (c) edge [dotted, very thick] node{$m-2$} (d);
                \path (l) edge [dotted, very thick] (a);
                \path (g) edge [dotted, very thick] (e);
                \draw[edge] (l)--(g);
                \draw[edge] (a)--(b);
                \path (l) edge [bend left = 70, dotted, very thick, ->, color = blue!50] node{$1$} (b);
                \path (g) edge [bend right, dotted, very thick, ->, color = blue!50] node{$2$} (a);
                \path (d) edge [bend right, dotted, very thick, ->, color = blue!50] node{$3$} (e);
            \end{tikzpicture}
            \caption{If there are two disjoint fire exits, we can move $i, j$ out of the way while still leaving a fire exit open for $p_1$.}
            \label{fig_9}
        \end{figure}
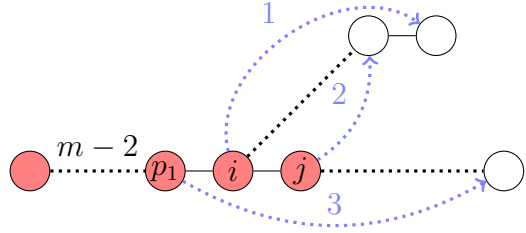
        \item If all fire exits from $y_{m}$ go through $y_{m - 1}$ and they all pass through one empty spot beyond $y_{m - 1}$, then we first bring all empty spots to be as close to $y_{m - 1}$ as possible. As in the proof of \Cref{prop_populate_cycle}, these empty spots cannot be in one path and must fork. We can bring $i$ to one of the empty spots and evacuate the person at $y_{m - 2}$. 
        \begin{figure}[ht]
            \centering
            \begin{tikzpicture}[node distance = 2 cm, auto, scale = 0.9]
                \tikzstyle{spot} = [circle, draw, inner sep = 0pt]
                \tikzstyle{edge} = [-]
                \tikzstyle{person} = [circle, draw, fill = red!50, inner sep = 0pt]

                \node[person](a) at (1, 0) {\phantom{$f_1$}};
                \node[person](d) at (3, 0) {\phantom{$f_1$}};
                \node[person](g) at (5, -1) {\phantom{$f_1$}};
                \node[person](l) at (4, 0) {\phantom{$f_1$}};
                \node[spot](h) at (5, 0) {\phantom{$f_1$}};
                \node[spot](i) at (8, 0) {\phantom{$f_1$}};
                \node[spot](j) at (9, 1) {\phantom{$f_1$}};
                \node[spot](k) at (9, -1) {\phantom{$f_1$}};
                
                \node () at (d) {$p_1$};
                \node () at (g) {$j$};
                \node () at (l) {$i$};

                \path (a) edge [dotted, very thick] node{$m - 2$} (d);
                \draw[edge] (d)--(l)--(h);
                \draw[edge] (l)--(g);
                \draw[edge] (i)--(j);
                \draw[edge] (i)--(k);
                \path (h) edge [dotted, very thick] node{$\leq n - k_a - 2$} (i);
                \path (l) edge [bend left, dotted, very thick, ->, color = blue!50] node{$1$} (j);
                \path (d) edge [bend right = 40, dotted, very thick, ->, color = blue!50] node{$2$} (k);
            \end{tikzpicture}
            \caption{We can evacuate $p_1$ using the two empty spots beyond the fork. }
            \label{fig_10}
        \end{figure}
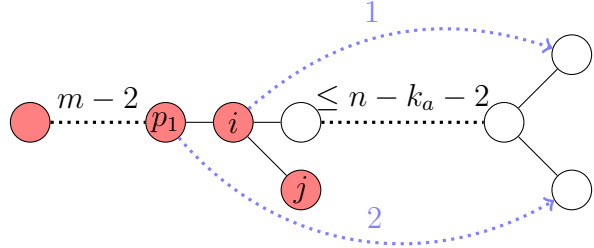
        \item If all fire exits from $y_{m-1}$ go through $y_{m}$ and they all pass through one empty spot beyond $y_{m}$, then we first bring all empty spots to be as close to $y_{m}$ as possible. Now, these empty spots cannot be in one path and must fork, and there must be at least three empty spots beyond the fork because $y_{m - 1}$ is also part of the bridge. We can bring $i, j$ to two of the empty spots and evacuate the person at $y_{m - 2}$ to the third. 
        \begin{figure}[ht]
            \centering
            \begin{tikzpicture}[node distance = 2 cm, auto, scale = 0.9]
                \tikzstyle{spot} = [circle, draw, inner sep = 0pt]
                \tikzstyle{edge} = [-]
                \tikzstyle{person} = [circle, draw, fill = red!50, inner sep = 0pt]

                \node[person](a) at (1, 0) {\phantom{$f_1$}};
                \node[person](d) at (3, 0) {\phantom{$f_1$}};
                \node[person](g) at (5, 0) {\phantom{$f_1$}};
                \node[person](l) at (4, 0) {\phantom{$f_1$}};
                \node[spot](h) at (6, 0) {\phantom{$f_1$}};
                \node[spot](i) at (9, 0) {\phantom{$f_1$}};
                \node[spot](j) at (10, 1) {\phantom{$f_1$}};
                \node[spot](k) at (10, -1) {\phantom{$f_1$}};
                \node[spot](n) at (11, 1) {\phantom{$f_1$}};
                
                \node () at (d) {$p_1$};
                \node () at (g) {$j$};
                \node () at (l) {$i$};

                \path (a) edge [dotted, very thick] node{$m - 2$} (d);
                \draw[edge] (d)--(l)--(g)--(h);
                \draw[edge] (i)--(j)--(n);
                \draw[edge] (i)--(k);
                \path (h) edge [dotted, very thick] node{$\leq n - k_t - 3$} (i);
                \path (g) edge [bend left = 40, dotted, very thick, ->, color = blue!50] node{$1$} (n);
                \path (l) edge [bend left, dotted, very thick, ->, color = blue!50] node{$2$} (j);
                \path (d) edge [bend right = 20, dotted, very thick, ->, color = blue!50] node{$3$} (k);
            \end{tikzpicture}
            \caption{Since $y_{m-1}$ is also a part of the bridge, there must be at least three empty spots beyond the fork.}
            \label{fig_11}
        \end{figure}
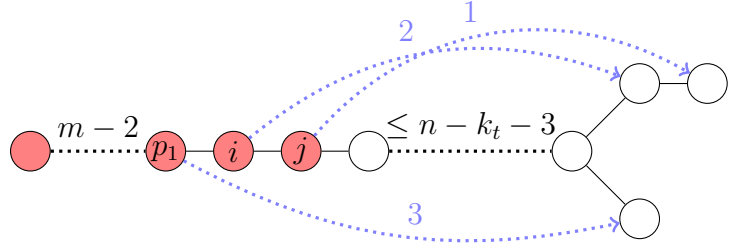
    \end{enumerate}
\end{proof}

With this lemma, we can directly prove the following.
\begin{proposition}\label{prop_snake_tongue}
    Suppose $X$ and $K_{k_1, \ldots, k_t}$ satisfy the same assumptions as in \Cref{thm_tree_reduction}. Let $s_a + 1, s_a + 2 \in S_a$ be two vertices of $K_{k_1, \ldots, k_t}$ of the same partition class and $\sigma$ a bijection with $\{\sigma^{-1}(s_a + 1), \sigma^{-1}(s_a + 2)\} \in E(X)$. Then there exists a snake tongue subgraph $Y \subset X$ with at most $n - k_t + 2$ vertices and a bijection $\sigma'$ in the same connected component of $\FS(X, K_{k_1, \ldots, k_t})$ as $\sigma$ such that $\sigma'^{-1}(S_a) \cap V(Y) = \{\sigma^{-1}(s_a + 1), \sigma^{-1}(s_a + 2)\}$ and $|\sigma'^{-1}(S_b) \cap V(Y)| > 0$ for at least two values of $b \neq a$.
\end{proposition}
\begin{proof}
    By \Cref{lem_tree_move_anywhere}, we can move $s_a + 1, s_a + 2$ to a leaf in $X$ and its neighbor. Now we move all empty spots so that they are as close to $s_a + 1$ and $s_a + 2$ as possible. The empty spots again must fork, and since we can move empty spots of different partition classes through each other, we get the desired snake tongue graph $Y$. 
\end{proof}

Finally, we prove that we can sort things out in a ``Y''-shaped graph. 
\begin{proposition}\label{prop_y_exchangeable}
    Suppose $Y$ is a snake-tongue graph with vertices $(t_1, \ldots, t_\ell, u_1, u_2)$ with $\ell \geq 2$. Suppose $t \geq 2$, $1 \leq k_1 \leq \cdots \leq k_t$ and $\sum_{i = 1}^{t} = \ell$. Let $K_{2, k_1, \ldots, k_t}$ have vertex set $[n]$ with the first partition class $\{1, 2\}$. Then for any bijection $\sigma$, $1$ and $2$ are $(Y, K_{2, k_1, \ldots, k_t})$-exchangeable from $\sigma$. 
\end{proposition}
\begin{proof}
    For any $\sigma$, we can first perform friendly swaps so that $\{1, 2\}$ occupy $\{u_1, u_2\}$ (not necessarily in this order). We can also perform friendly swaps so that $u_{\ell - 1}$ and $u_{\ell}$ are occupied by $i$, $j$ of different partition classes. So we can assume without loss of generality (by \Cref{lem_u_v_equivalent}) that $\sigma^{-1}(1) = u_1$, $\sigma^{-1}(2) = u_2$, and $u_{\ell - 1}$ and $u_{\ell}$ are occupied by $i, j$ of different partition classes. Now it is easy to check that $1$ and $2$ are $(Y, K_{2, k_1, \ldots, k_t})$-exchangeable from $\sigma$ (see \Cref{fig_13}). 
    \begin{figure}[ht]
        \centering
        \begin{tikzpicture}[node distance = 2 cm, auto, scale = 0.9]
            \tikzstyle{spot} = [circle, draw, inner sep = 1pt]
            \tikzstyle{edge} = [-]
            \tikzstyle{namedperson} = [circle, draw, fill = red!50, inner sep = 1pt]
            \tikzstyle{specialspot} = [circle, draw, fill = yellow!50, inner sep = 1pt]
            
            \node[spot](1) at (0, 0) {\phantom{$2$}};
            \node[specialspot](2) at (2, 0) {\phantom{$2$}};
            \node[namedperson](3) at (3, 1) {$1$};
            \node[namedperson](4) at (3, -1) {$2$};

            \draw[edge] (1)--(2);
            \draw[edge] (2)--(3);
            \draw[edge] (2)--(4);

            \path (3) edge [bend right, dotted, very thick, ->, color = blue!50] (1);

            \node[namedperson](1) at (0+4, 0) {$1$};
            \node[spot](2) at (2+4, 0) {\phantom{$2$}};
            \node[specialspot](3) at (3+4, 1) {\phantom{$2$}};
            \node[namedperson](4) at (3+4, -1) {$2$};

            \draw[edge] (1)--(2);
            \draw[edge] (2)--(3);
            \draw[edge] (2)--(4);

            \path (4) edge [bend left, dotted, very thick, ->, color = blue!50] (3);

            \node[namedperson](1) at (0+8, 0) {$1$};
            \node[specialspot](2) at (2+8, 0) {\phantom{$2$}};
            \node[namedperson](3) at (3+8, 1) {$2$};
            \node[spot](4) at (3+8, -1) {\phantom{$2$}};

            \draw[edge] (1)--(2);
            \draw[edge] (2)--(3);
            \draw[edge] (2)--(4);

            \path (1) edge [bend right, dotted, very thick, ->, color = blue!50] (4);

            \node[specialspot](1) at (0+12, 0) {\phantom{$2$}};
            \node[spot](2) at (2+12, 0) {\phantom{$2$}};
            \node[namedperson](3) at (3+12, 1) {$2$};
            \node[namedperson](4) at (3+12, -1) {$1$};

            \draw[edge] (1)--(2);
            \draw[edge] (2)--(3);
            \draw[edge] (2)--(4);

            \path (1) edge [bend left, dotted, very thick, <->, color = blue!50] (2);
            
        \end{tikzpicture}
        \caption{The way we exchange $1$ and $2$ in the ``fork'' end of $Y$. It is crucial that we have two types of empty spots so we can swap them back in the last step.}
        \label{fig_12}
    \end{figure}
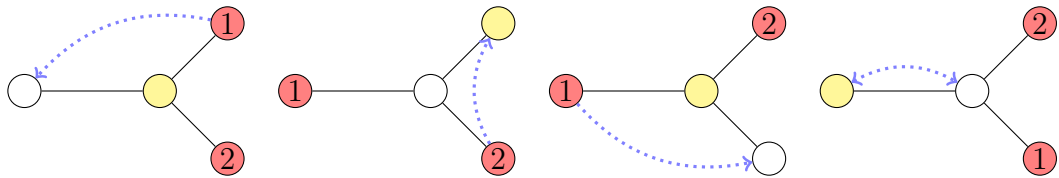
\end{proof}
\begin{proof}[Proof of \Cref{thm_tree_reduction}]
    This follows from \Cref{prop_snake_tongue}, \Cref{prop_y_exchangeable}, \Cref{lem_u_v_equivalent}, and \Cref{lem_localization}. 
\end{proof}

We are ready to put everything together.
\begin{proof}[Proof of \Cref{thm_main}]
    Case (1) is Wilson's theorem from \cite{wilson1974}. 
    
    Case (2) follows from case (1) of \Cref{thm_multi_connected} and \Cref{thm_multi_disconnected}.

    Case (3) follows from case (2) of \Cref{thm_multi_connected}, \Cref{thm_multi_disconnected}, and \Cref{thm_tree}.
    
    Case (4) follows from case (2) of \Cref{thm_multi_connected}, \Cref{thm_multi_disconnected}, \Cref{thm_tree}, and \Cref{lem_cycle_components}.
\end{proof}

\section{Exactly two components in $\FS(X, K_{k, n-k})$}\label{sec_two_comps}
In this section, we prove \Cref{thm_two_comps}. We need the following lemma of Alon, Defant, and Kravitz.
\begin{lemma}[\cite{alon2023}]\label{lem_two_components}
    For any $n \geq 5$ and $k, \ell \geq 2$, $\FS(K_{k, n - k}, K_{\ell, n - \ell})$ has exactly two connected components. Furthermore, the connected component containing the bijection $\sigma$ is determined by the parity of $\operatorname{sgn}(\sigma) + |\sigma(\{1, \ldots, k\}) \cap \{1, \ldots, \ell\}|$. 
\end{lemma}

In this section, we assume that $X$ is a connected bipartite graph on $n \geq 5$ vertices that is not a cycle. By \Cref{lem_bipartite_disconnected}, $\FS(X, K_{k, n - k})$ has at least two connected components. On the other hand, if $X$ also does not contain a $k$-bridge and  $n \geq 2k$, then Case (3) of \Cref{thm_main} implies that $\FS(X, K_{1, k-1, n - k})$ is connected. 

Suppose that $K_{k, n - k}$ has vertex bipartition $\{1, \ldots, k\}$, $\{k + 1, \ldots, n\}$. Let $K'$ be the graph where we add the edge $\{1, 2\}$ to $K_{k, n - k}$. We first prove the ``only if'' direction of \Cref{thm_two_comps}.
\begin{proposition}\label{prop_more_than_two_comp}
    Suppose $n \geq 5$ and $n \geq 2k \geq 4$. Let $X$ be a connected bipartite graph on $n$ vertices that is not a cycle. The following hold:
    \begin{enumerate}
        \item If $X$ contains a $k$-bridge, then $\FS(X, K_{k, n - k})$ has more than two connected components. 
        \item If $X$ and $k$ are one of the three exceptions in \Cref{thm_two_comps}, then $\FS(X, K_{k, n - k})$ has exactly six connected components. 
    \end{enumerate}
\end{proposition}
\begin{proof}
    For the three exceptions (see \Cref{fig_exceptions}), we can check by brute force (e.g., writing a python program) that $\FS(X, K_{k, n - k})$ has exactly six connected components. 
    
    Now suppose $X$ contains a $k$-bridge and let $K'$ be as above. Then by \Cref{thm_main}, $\FS(X, K_{1, k-1, n - k})$ is disconnected. By \Cref{lem_spanning_subgraph}, this implies that $\FS(X, K')$ is disconnected. On the other hand, for any bijection $\sigma$ where $\{\sigma^{-1}(1), \sigma^{-1}(2)\} \in E(X)$, we can swap $1$ and $2$ to get $\sigma'$ such that $\operatorname{sgn}(\sigma) + |\sigma(\{1, \ldots, k\}) \cap \{1, \ldots, \ell\}|$ and $\operatorname{sgn}(\sigma') + |\sigma'(\{1, \ldots, k\}) \cap \{1, \ldots, \ell\}|$ have different parity. By \Cref{lem_two_components} and \Cref{lem_spanning_subgraph}, $\sigma$ and $\sigma'$ are in different connected components of $\FS(X, K_{k, n - k})$. Therefore, $\FS(X, K_{k, n - k})$ must have strictly more connected components than $\FS(X, K_{1, k-1, n - k})$. 
\end{proof}

Now we attack the ``if'' direction. 
\begin{theorem}\label{thm_K'}
    Suppose $n \geq 5$ and $n \geq 2k \geq 4$. Let $X$ be a connected bipartite graph on $n$ vertices that is not a cycle. Let $K'$ be the graph where we add the edge $\{1, 2\}$ to $K_{k, n - k}$. If $X$ is not one of the three exceptions and does not contain a $k$-bridge, then $\FS(X, K')$ is connected. 
\end{theorem}
With this, we can easily prove \Cref{thm_two_comps}.
\begin{proof}[Proof of \Cref{thm_two_comps} assuming \Cref{thm_K'}]
    The ``only if'' direction follows from \Cref{prop_more_than_two_comp}. 
    
    For the ``if'' direction, since $K_{k, n - k}$ is only missing the edge $\{1, 2\}$ from $K'$ and they have the same neighbors in $K_{k, n - k}$, we can reach any bijection up to a permutation of $(1 \; 2)$. Indeed, suppose some sequence of friendly swaps takes $\sigma$ to $\sigma'$ in $\FS(X, K')$, then we can perform the same sequence of friendly swaps over edges of $X$ but skipping those where $1$ and $2$ are swapped. Then we will end up with some $\sigma''$ that is in the same connected component as $\sigma$ in $\FS(X, K_{k, n-k})$ and is either equal to $\sigma'$ or $(1 \; 2) \circ \sigma'$. Thus, $\FS(X, K_{k, n-k})$ has at most two connected components. Since it also cannot be connected, we are done. 
\end{proof}
If $k = 2$, we are already done since $K'$ is precisely $K_{1, 1, n - 2}$. Now we assume $k \geq 3$. 

As in \Cref{sec_not_tree} and \Cref{sec_tree}, we can use \Cref{prop_exchange} to reduce \Cref{thm_K'} to the following.
\begin{theorem}\label{thm_two_comps_reduction}
    Let $X$ and $K'$ be as above, and suppose that $X$ satisfies the assumptions in \Cref{thm_K'}. For any $j \in \{3, \ldots, k\}$, if $\sigma: V(X) \to V(K')$ satisfies $\{\sigma^{-1}(1), \sigma^{-1}(j)\} \in E(X)$, then the vertices $1$ and $j$ are $(X, K')$-exchangeable from $\sigma$. 
\end{theorem}
\begin{proof}[Proof of \Cref{thm_two_comps_reduction}]
    By symmetry, it suffices to consider $j = 3$. Since the case $k = 2$ holds by the connectedness of $\FS(X, K_{1, k-1, n - k})$, we start with $k > 2$. We treat elements of $\{1, \ldots, k\}$ as people and others as empty spots. 

    \textbf{Case I: $X$ is not a tree.} In this case, we can find a minimal even cycle $C \subset X$ and $|V(C)| < n$. By \Cref{prop_populate_cycle}, there is a bijection $\sigma'$ which can be obtained from $\sigma$ using a sequence of friendly swaps that does not include the swap across $\{1, 2\}$ such that $\sigma'^{-1}(1), \sigma'^{-1}(2), \sigma'^{-1}(3) \in V(C)$ and $|V(C) \cap \sigma'^{-1}(\{k + 1, \ldots, n\})| \geq 1$. If there is only one empty spot in $C$, then we can bring another empty spot to a vertex adjacent to $C$.
    
    By \Cref{thm_stopwatch}, we can reach half of the orderings of people and empty spots in $C$, but since we are also allowed to swap $1$ and $2$, we would be able to get all the possible ordering; in particular, we can exchange $1$ and $3$.

    Since we did not swap $1$ and $2$ when going from $\sigma$ to $\sigma'$, \Cref{lem_u_v_equivalent} still applies and thus $1$ and $3$ are $(X, K')$-exchangeable from $\sigma$. 
    
    \textbf{Case II: $X$ is a tree and contains a vertex of degree at least $4$.} Let $u$ be a vertex of degree at least $4$ and let $v_1, \ldots, v_m$ where $m \geq 4$ be its neighbors. For a neighbor $v_i$ of $u$, we say that the \blue{\emph{branch}} of $v_i$ is the component of $X$ containing $v_i$ when the edge $\{u, v_i\}$ is removed. 
    
    By \Cref{lem_tree_move_anywhere}, we can perform a sequence of friendly swaps to move $1$ and $2$ to occupy $u$ and $v_1$ (not necessarily in this order) without using the swap across $\{1, 2\}$. Without loss of generality suppose $1$ is on $u$. Now we can bring all empty spots so that they are as close to $1$ as possible without moving $1$ or $2$. It is not hard to check that it is possible to move the empty spots around so that $v_2$, $v_3$, and $v_4$ are all occupied by empty spots. Now suppose $x_1 = u, x_2, \ldots, x_\ell$ is the path from $u$ to where $3$ is. If $\ell = 2$, this means $x_2 = v_5$ is another neighbor of $u$ (see \Cref{fig_13}). By Wilson's theorem (case (1) of \Cref{thm_main}) and \Cref{lem_localization}, $1$ and $3$ are exchangeable.

    \begin{figure}[ht]
            \centering
            \begin{tikzpicture}[node distance = 2 cm, auto, scale = 0.9]
            \tikzstyle{spot} = [circle, draw, inner sep = 0pt]
            \tikzstyle{edge} = [-]
            \tikzstyle{person} = [circle, draw, fill = red!50, inner sep = 0pt]
    
            \node[spot](a) at (0, 0) {\phantom{$f_1$}};
            \node[spot](b) at (1, 1) {\phantom{$f_1$}};
            \node[spot](c) at (1, -1) {\phantom{$f_1$}};
            \node[person](d) at (1, 0) {\phantom{$f_1$}};
            \node[person](e) at (2, 0.7) {\phantom{$f_1$}};
            \node[person](f) at (2, -0.7) {\phantom{$f_1$}};
    
            \node () at (d) {$1$};
            \node () at (e) {$2$};
            \node () at (f) {$3$};
    
            \draw[edge] (a)--(d)--(e);
            \draw[edge] (b)--(d)--(c);
            \draw[edge] (d)--(f);
            
            \end{tikzpicture}
            \caption{By Wilson's theorem, $1$ and $3$ are exchangeable from this configuration.}
            \label{fig_13}
    \end{figure}
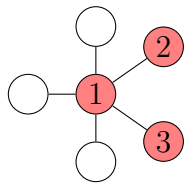

    Now suppose $\ell > 2$. If $x_2 \neq v_1$, we can just move the person at $u$ into $x_2$ and the person at $v_1$ into $u$. Thus, we may assume that $x_2 = v_1$, $1$ is at $x_1 = u$, and $2$ is at $x_2 = v_1$. Our goal is to bring $3$ in so that the vertices $\{u, v_1, v_2, v_3, v_4, x_2\}$ is occupied by $1, 2, 3$ and three empty spots. By brute force, we can show that the friends and strangers graph on the induced subgraph of $X$ on $\{u, v_1, v_2, v_3, v_4, x_2\}$ and the induced subgraph of $K'$ on $\{1, 2, 3, k+1, k+2, k+3\}$ is connected (see \Cref{fig_14}). Then by \Cref{lem_localization}, $1$ and $3$ are exchangeable from $\sigma$. 

    \begin{figure}[ht]
            \centering
            \begin{tikzpicture}[node distance = 2 cm, auto, scale = 0.9]
            \tikzstyle{spot} = [circle, draw, inner sep = 0pt]
            \tikzstyle{edge} = [-]
            \tikzstyle{person} = [circle, draw, fill = red!50, inner sep = 0pt]
    
            \node[spot](a) at (0, 0) {\phantom{$f_1$}};
            \node[spot](b) at (1, 1) {\phantom{$f_1$}};
            \node[spot](c) at (1, -1) {\phantom{$f_1$}};
            \node[person](d) at (1, 0) {\phantom{$f_1$}};
            \node[person](e) at (2, 0) {\phantom{$f_1$}};
            \node[person](f) at (3, 0) {\phantom{$f_1$}};
    
            \node () at (d) {$1$};
            \node () at (e) {$2$};
            \node () at (f) {$3$};
            \node () at (b) {$4$};
            \node () at (a) {$5$};
            \node () at (c) {$6$};
    
            \draw[edge] (a)--(d)--(e)--(f);
            \draw[edge] (b)--(d)--(c);
            
            \end{tikzpicture}
            \caption{We can show by brute force that any configuration can be reached within this subgraph. In particular, we can always exchange $1$ and $3$,}
            \label{fig_14}
    \end{figure}
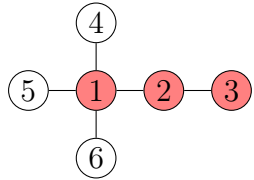
    
    If $v_2, v_3, v_4$ are not the only vertices with empty spots not on the branch of $v_1$, then we can move $1$ and $2$ out of the way and evacuate the person at $x_3$ to fill up that extra empty spot (see \Cref{fig_15}). Moving $1$ and $2$ back, we have brought $3$ one step closer.
        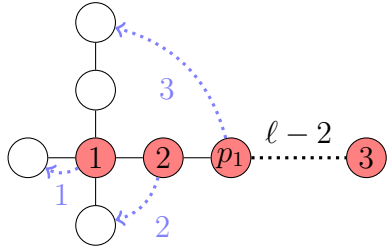
\begin{figure}[ht]
            \centering
            \begin{tikzpicture}[node distance = 2 cm, auto, scale = 0.9]
            \tikzstyle{spot} = [circle, draw, inner sep = 0pt]
            \tikzstyle{edge} = [-]
            \tikzstyle{person} = [circle, draw, fill = red!50, inner sep = 0pt]
    
            \node[spot](a) at (0, 0) {\phantom{$f_1$}};
            \node[spot](b) at (1, 1) {\phantom{$f_1$}};
            \node[spot](c) at (1, -1) {\phantom{$f_1$}};
            \node[person](d) at (1, 0) {\phantom{$f_1$}};
            \node[person](e) at (2, 0) {\phantom{$f_1$}};
            \node[person](f) at (3, 0) {\phantom{$f_1$}};
            \node[person](g) at (5, 0) {\phantom{$f_1$}};
            \node[spot](h) at (1, 2) {\phantom{$f_1$}};
    
            \node () at (d) {$1$};
            \node () at (e) {$2$};
            \node () at (f) {$p_1$};
            \node () at (g) {$3$};
    
            \path (f) edge [dotted, very thick] node{$\ell -2$} (g);
            \draw[edge] (a)--(d)--(e)--(f);
            \draw[edge] (h)--(b)--(d)--(c);
            
            \path (d) edge [bend left, dotted, very thick, ->, color = blue!50] node{$1$} (a);
            \path (e) edge [bend left, dotted, very thick, ->, color = blue!50] node{$2$} (c);
            \path (f) edge [bend right, dotted, very thick, ->, color = blue!50] node{$3$} (h);
            \end{tikzpicture}
            \caption{If there are more than three empty spots outside the branch containing $3$, we can move $1$ and $2$ out of the way to evacuate $p_1$ from the path. It is also possible that this extra empty spot is another neighbor of $u$. The same procedure applies.}
            \label{fig_15}
        \end{figure}
        
    Now we may assume that all but three empty spots are in the branch of $v_1$. We move $1$ into $v_3$. If we cut off the branch of $v_3$, we get a tree that still does not contain a $k$-bridge. In particular, it does not contain an $(n - k)$-bridge. By \Cref{lem_tree_move_anywhere}, we can bring $2$ and $3$ to occupy $u$ and $v_1$. It is now not difficult to show that we can move things around to have empty spots in $v_2, v_4, x_2$. 
    
    \textbf{Case III: $X$ is a tree and contains two vertices of degree $3$.} We can always find a vertex of degree $3$ so that all other vertices of degree $3$ lie on one branch from this vertex. 
    
    Suppose $x_1, \ldots, x_\ell$ is a branch that is a path from a leaf $x_1$ to this vertex $x_\ell$. Suppose $y_1 = x_\ell, \ldots, y_m$ is the path to another vertex of degree $3$ closest to $x_\ell$. Let $u_1 \neq y_2$ be the other neighbor of $x_\ell$ and let $u_2, u_3 \neq y_{m - 1}$ be the other two neighbors of $y_m$. Let $D$ be the induced subgraph of $X$ with $V(D) = \{y_1, \ldots, y_m, x_{\ell -1}, u_1, u_2, u_3\}$. We claim that if $1$ is at $x_{\ell -1}$, $2$ is at $y_1$, $3$ is at $u_1$, and the rest of the vertices of $D$ are occupied by empty spots, then we can exchange $1$ and $3$ within this ``dog bone'' subgraph (see \Cref{fig_16}). 

    \begin{figure}[ht]
            \centering
            \begin{tikzpicture}[node distance = 2 cm, auto, scale = 0.9]
            \tikzstyle{spot} = [circle, draw, inner sep = 0pt]
            \tikzstyle{edge} = [-]
            \tikzstyle{person} = [circle, draw, fill = red!50, inner sep = 0pt]
    
            \node[person](a) at (0, 1) {\phantom{$f_1$}};
            \node[person](b) at (0, -1) {\phantom{$f_1$}};
            \node[person](c) at (1, 0) {\phantom{$f_1$}};
            \node[spot](g) at (2, 0) {\phantom{$f_1$}};
            \node[spot](d) at (4, 0) {\phantom{$f_1$}};
            \node[spot](e) at (5, 1) {\phantom{$f_1$}};
            \node[spot](f) at (5, -1) {\phantom{$f_1$}};
    
            \node () at (b) {$1$};
            \node () at (c) {$2$};
            \node () at (a) {$3$};
    
            \draw[edge] (a)--(c)--(b);
            \draw[edge] (c)--(g);
            \draw[edge] (e)--(d)--(f);
            \path (g) edge [dotted, very thick] node{$m-1$} (d);
            \end{tikzpicture}
            \caption{We claim that $1$ and $3$ are exchangeable from this configuration.}
            \label{fig_16}
    \end{figure}
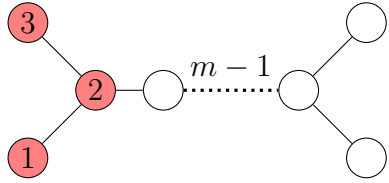

    We proof the claim explicitly by giving a sequence of friendly swaps (see \Cref{fig_20}). 

    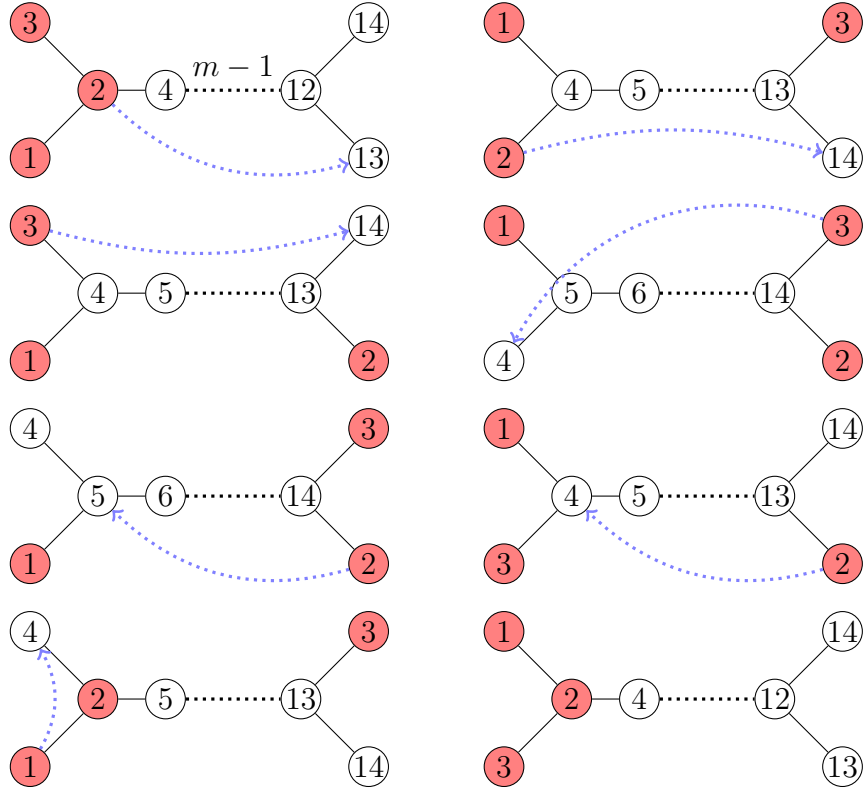
\begin{figure}[ht]
            \centering
            \begin{tikzpicture}[node distance = 2 cm, auto, scale = 0.9]
            \tikzstyle{spot} = [circle, draw, inner sep = 0pt]
            \tikzstyle{edge} = [-]
            \tikzstyle{person} = [circle, draw, fill = red!50, inner sep = 0pt]
    
            \node[person](a) at (0, 1) {\phantom{$f_1$}};
            \node[person](b) at (0, -1) {\phantom{$f_1$}};
            \node[person](c) at (1, 0) {\phantom{$f_1$}};
            \node[spot](g) at (2, 0) {\phantom{$f_1$}};
            \node[spot](d) at (4, 0) {\phantom{$f_1$}};
            \node[spot](e) at (5, 1) {\phantom{$f_1$}};
            \node[spot](f) at (5, -1) {\phantom{$f_1$}};
    
            \node () at (b) {$1$};
            \node () at (c) {$2$};
            \node () at (a) {$3$};
            \node () at (g) {$4$};
            \node () at (d) {$12$};
            \node () at (f) {$13$};
            \node () at (e) {$14$};
    
            \draw[edge] (a)--(c)--(b);
            \draw[edge] (c)--(g);
            \draw[edge] (e)--(d)--(f);
            \path (g) edge [dotted, very thick] node{$m-1$} (d);
            \path (c) edge [bend right, dotted, very thick, ->, color = blue!50] (f);

            \node[person](a) at (0, 1-3) {\phantom{$f_1$}};
            \node[person](b) at (0, -1-3) {\phantom{$f_1$}};
            \node[spot](c) at (1, 0-3) {\phantom{$f_1$}};
            \node[spot](g) at (2, 0-3) {\phantom{$f_1$}};
            \node[spot](d) at (4, 0-3) {\phantom{$f_1$}};
            \node[spot](e) at (5, 1-3) {\phantom{$f_1$}};
            \node[person](f) at (5, -1-3) {\phantom{$f_1$}};
    
            \node () at (b) {$1$};
            \node () at (c) {$4$};
            \node () at (a) {$3$};
            \node () at (g) {$5$};
            \node () at (d) {$13$};
            \node () at (f) {$2$};
            \node () at (e) {$14$};
    
            \draw[edge] (a)--(c)--(b);
            \draw[edge] (c)--(g);
            \draw[edge] (e)--(d)--(f);
            \path (g) edge [dotted, very thick] (d);
            \path (a) edge [bend right = 15, dotted, very thick, ->, color = blue!50] (e);

            \node[spot](a) at (0, 1-6) {\phantom{$f_1$}};
            \node[person](b) at (0, -1-6) {\phantom{$f_1$}};
            \node[spot](c) at (1, 0-6) {\phantom{$f_1$}};
            \node[spot](g) at (2, 0-6) {\phantom{$f_1$}};
            \node[spot](d) at (4, 0-6) {\phantom{$f_1$}};
            \node[person](e) at (5, 1-6) {\phantom{$f_1$}};
            \node[person](f) at (5, -1-6) {\phantom{$f_1$}};
    
            \node () at (b) {$1$};
            \node () at (c) {$5$};
            \node () at (a) {$4$};
            \node () at (g) {$6$};
            \node () at (d) {$14$};
            \node () at (f) {$2$};
            \node () at (e) {$3$};
    
            \draw[edge] (a)--(c)--(b);
            \draw[edge] (c)--(g);
            \draw[edge] (e)--(d)--(f);
            \path (g) edge [dotted, very thick] (d);
            \path (f) edge [bend left, dotted, very thick, ->, color = blue!50] (c);

            \node[spot](a) at (0, 1-9) {\phantom{$f_1$}};
            \node[person](b) at (0, -1-9) {\phantom{$f_1$}};
            \node[person](c) at (1, 0-9) {\phantom{$f_1$}};
            \node[spot](g) at (2, 0-9) {\phantom{$f_1$}};
            \node[spot](d) at (4, 0-9) {\phantom{$f_1$}};
            \node[person](e) at (5, 1-9) {\phantom{$f_1$}};
            \node[spot](f) at (5, -1-9) {\phantom{$f_1$}};
    
            \node () at (b) {$1$};
            \node () at (c) {$2$};
            \node () at (a) {$4$};
            \node () at (g) {$5$};
            \node () at (d) {$13$};
            \node () at (f) {$14$};
            \node () at (e) {$3$};
    
            \draw[edge] (a)--(c)--(b);
            \draw[edge] (c)--(g);
            \draw[edge] (e)--(d)--(f);
            \path (g) edge [dotted, very thick] (d);
            \path (b) edge [bend right, dotted, very thick, ->, color = blue!50] (a);

            \node[person](a) at (0+7, 1) {\phantom{$f_1$}};
            \node[person](b) at (0+7, -1) {\phantom{$f_1$}};
            \node[spot](c) at (1+7, 0) {\phantom{$f_1$}};
            \node[spot](g) at (2+7, 0) {\phantom{$f_1$}};
            \node[spot](d) at (4+7, 0) {\phantom{$f_1$}};
            \node[person](e) at (5+7, 1) {\phantom{$f_1$}};
            \node[spot](f) at (5+7, -1) {\phantom{$f_1$}};
    
            \node () at (b) {$2$};
            \node () at (c) {$4$};
            \node () at (a) {$1$};
            \node () at (g) {$5$};
            \node () at (d) {$13$};
            \node () at (f) {$14$};
            \node () at (e) {$3$};
    
            \draw[edge] (a)--(c)--(b);
            \draw[edge] (c)--(g);
            \draw[edge] (e)--(d)--(f);
            \path (g) edge [dotted, very thick] (d);
            \path (b) edge [bend left = 15, dotted, very thick, ->, color = blue!50] (f);

            \node[person](a) at (0+7, 1-3) {\phantom{$f_1$}};
            \node[spot](b) at (0+7, -1-3) {\phantom{$f_1$}};
            \node[spot](c) at (1+7, 0-3) {\phantom{$f_1$}};
            \node[spot](g) at (2+7, 0-3) {\phantom{$f_1$}};
            \node[spot](d) at (4+7, 0-3) {\phantom{$f_1$}};
            \node[person](e) at (5+7, 1-3) {\phantom{$f_1$}};
            \node[person](f) at (5+7, -1-3) {\phantom{$f_1$}};
    
            \node () at (b) {$4$};
            \node () at (c) {$5$};
            \node () at (a) {$1$};
            \node () at (g) {$6$};
            \node () at (d) {$14$};
            \node () at (f) {$2$};
            \node () at (e) {$3$};
    
            \draw[edge] (a)--(c)--(b);
            \draw[edge] (c)--(g);
            \draw[edge] (e)--(d)--(f);
            \path (g) edge [dotted, very thick] (d);
            \path (e) edge [bend right = 40, dotted, very thick, ->, color = blue!50] (b);

            \node[person](a) at (0+7, 1-6) {\phantom{$f_1$}};
            \node[person](b) at (0+7, -1-6) {\phantom{$f_1$}};
            \node[spot](c) at (1+7, 0-6) {\phantom{$f_1$}};
            \node[spot](g) at (2+7, 0-6) {\phantom{$f_1$}};
            \node[spot](d) at (4+7, 0-6) {\phantom{$f_1$}};
            \node[spot](e) at (5+7, 1-6) {\phantom{$f_1$}};
            \node[person](f) at (5+7, -1-6) {\phantom{$f_1$}};
    
            \node () at (b) {$3$};
            \node () at (c) {$4$};
            \node () at (a) {$1$};
            \node () at (g) {$5$};
            \node () at (d) {$13$};
            \node () at (f) {$2$};
            \node () at (e) {$14$};
    
            \draw[edge] (a)--(c)--(b);
            \draw[edge] (c)--(g);
            \draw[edge] (e)--(d)--(f);
            \path (g) edge [dotted, very thick] (d);
            \path (f) edge [bend left, dotted, very thick, ->, color = blue!50] (c);

            \node[person](a) at (0+7, 1-9) {\phantom{$f_1$}};
            \node[person](b) at (0+7, -1-9) {\phantom{$f_1$}};
            \node[person](c) at (1+7, 0-9) {\phantom{$f_1$}};
            \node[spot](g) at (2+7, 0-9) {\phantom{$f_1$}};
            \node[spot](d) at (4+7, 0-9) {\phantom{$f_1$}};
            \node[spot](e) at (5+7, 1-9) {\phantom{$f_1$}};
            \node[spot](f) at (5+7, -1-9) {\phantom{$f_1$}};
    
            \node () at (b) {$3$};
            \node () at (c) {$2$};
            \node () at (a) {$1$};
            \node () at (g) {$4$};
            \node () at (d) {$12$};
            \node () at (f) {$13$};
            \node () at (e) {$14$};
    
            \draw[edge] (a)--(c)--(b);
            \draw[edge] (c)--(g);
            \draw[edge] (e)--(d)--(f);
            \path (g) edge [dotted, very thick] (d);
            
            \end{tikzpicture}
            \caption{We illustrate the claim with $m = 10$.}
            \label{fig_20}
    \end{figure}
    
    Now that we have proved the claim, it suffices to move people into the configuration of the claim. By \Cref{lem_tree_move_anywhere}, we can bring $1$ and $2$ to occupy $x_1$ and $x_2$ (without loss of generality, assume $1$ is at $x_1$). Since there is no $k$-bridge, $\ell \leq k$. We want to move $1$ and $2$ to occupy $x_{\ell -1}$ and $x_{\ell}$ so that there are at least $k$ empty spots in $V(X) \setminus \{x_1, \ldots, x_{\ell -2}\}$ and $3$ is also not at any vertex in $\{x_1, \ldots, x_{\ell -2}\}$. If $\ell = 2$, we are already done. Assume $\ell > 2$. If there are no people other than $1, 2, 3$ in $V(X) \setminus \{x_1, \ldots, x_{\ell - 2}\}$, we can directly move $1$ and $2$ up and still have the conditions satisfied. 

    If $\ell = 3$, we can move empty spots to occupy $u_1, \ldots, y_1, \ldots, y_{m}$ since there are at least $k$ empty spots and $m \leq k - 1$. We can ``bury'' a person not in $\{1, 2, 3\}$ in $x_1$ and have $1$, $2$ at $x_2$ and $x_3$ (see \Cref{fig_17}). This way, $X$ with $x_1$ removed still has at least $k$ empty spots and contains $3$. 

    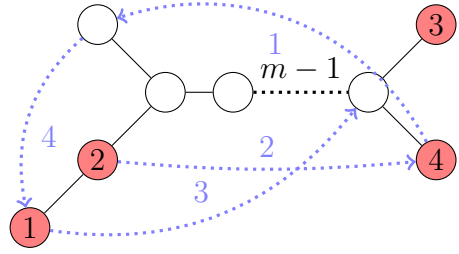
\begin{figure}[ht]
        \centering
        \begin{tikzpicture}[node distance = 2 cm, auto, scale = 0.9]
            \tikzstyle{spot} = [circle, draw, inner sep = 0pt]
            \tikzstyle{edge} = [-]
            \tikzstyle{person} = [circle, draw, fill = red!50, inner sep = 0pt]
    
            \node[person](x) at (-1,-2) {\phantom{$f_1$}};
            \node[spot](a) at (0, 1) {\phantom{$f_1$}};
            \node[person](b) at (0, -1) {\phantom{$f_1$}};
            \node[spot](c) at (1, 0) {\phantom{$f_1$}};
            \node[spot](g) at (2, 0) {\phantom{$f_1$}};
            \node[spot](d) at (4, 0) {\phantom{$f_1$}};
            \node[person](e) at (5, 1) {\phantom{$f_1$}};
            \node[person](f) at (5, -1) {\phantom{$f_1$}};
    
            \node () at (x) {$1$};
            \node () at (b) {$2$};
            \node () at (e) {$3$};
            \node () at (f) {$4$};
    
            \draw[edge] (a)--(c)--(b)--(x);
            \draw[edge] (c)--(g);
            \draw[edge] (e)--(d)--(f);
            \path (g) edge [dotted, very thick] node{$m-1$} (d);

            \path (f) edge [bend right = 40, dotted, very thick, ->, color = blue!50] node{$1$} (a);
            \path (b) edge [bend right = 5, dotted, very thick, ->, color = blue!50] node{$2$} (f);
            \path (x) edge [bend right, dotted, very thick, ->, color = blue!50] node{$3$} (d);
            \path (a) edge [bend right, dotted, very thick, ->, color = blue!50] node{$4$} (x);
        \end{tikzpicture}
        \caption{An example of how we can ``bury'' the person at $u_3$. Other cases can be worked out similarly.}
        \label{fig_17}
    \end{figure}

    Now assume $\ell > 3$. We can bring empty spots to occupy $\{x_3, \ldots, x_{\ell}, u_1, y_2\}$. Next we bring people as close to these empty spots as possible. Again, we can bury a person (not $3$) below $1$ and $2$ so that we can disregard them (see \Cref{fig_18} and \Cref{fig_19}). 

    \begin{figure}[ht]
        \centering
        \begin{tikzpicture}[node distance = 2 cm, auto, scale = 0.9]
            \tikzstyle{spot} = [circle, draw, inner sep = 0pt]
            \tikzstyle{edge} = [-]
            \tikzstyle{person} = [circle, draw, fill = red!50, inner sep = 0pt]

            \node[person](y) at (-2.5,-3.5) {\phantom{$f_1$}};
            \node[person](x) at (-1.5,-2.5) {\phantom{$f_1$}};
            \node[spot](a) at (0, 1) {\phantom{$f_1$}};
            \node[spot](b) at (-0.5, -1.5) {\phantom{$f_1$}};
            \node[spot](c) at (1, 0) {\phantom{$f_1$}};
            \node[spot](g) at (2, 0) {\phantom{$f_1$}};
            \node[person](d) at (3, 0) {\phantom{$f_1$}};
    
            \node () at (y) {$1$};
            \node () at (x) {$2$};
            \node () at (d) {$4$};
    
            \draw[edge] (a)--(c);
            \draw[edge] (b)--(x)--(y);
            \draw[edge] (c)--(g)--(d);
            \path (b) edge [dotted, very thick] node{$\ell-2$} (c);

            \path (d) edge [bend right, dotted, very thick, ->, color = blue!50] node{$1$} (a);
            \path (x) edge [bend right = 50, dotted, very thick, ->, color = blue!50] node{$2$} (d);
            \path (y) edge [bend right, dotted, very thick, ->, color = blue!50] node{$3$} (g);
            \path (a) edge [bend right, dotted, very thick, ->, color = blue!50] node{$4$} (y);
        \end{tikzpicture}
        \caption{An example of how we can bury the person next to $y_2$. Other cases can be worked out similarly.}
        \label{fig_18}
    \end{figure}
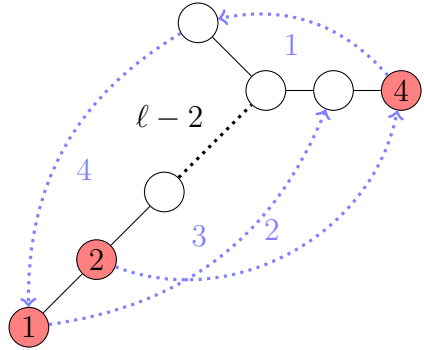
    \begin{figure}[ht]
        \centering
        \begin{tikzpicture}[node distance = 2 cm, auto, scale = 0.9]
            \tikzstyle{spot} = [circle, draw, inner sep = 0pt]
            \tikzstyle{edge} = [-]
            \tikzstyle{person} = [circle, draw, fill = red!50, inner sep = 0pt]

            \node[person](y) at (-2.5,-3.5) {\phantom{$f_1$}};
            \node[person](x) at (-1.5,-2.5) {\phantom{$f_1$}};
            \node[spot](a) at (0, 1) {\phantom{$f_1$}};
            \node[spot](b) at (-0.5, -1.5) {\phantom{$f_1$}};
            \node[spot](c) at (1, 0) {\phantom{$f_1$}};
            \node[spot](g) at (2, 0) {\phantom{$f_1$}};
            \node[person](d) at (3, 0) {\phantom{$f_1$}};
            \node[person](e) at (4, 0) {\phantom{$f_1$}};
    
            \node () at (y) {$1$};
            \node () at (x) {$2$};
            \node () at (d) {$3$};
            \node () at (e) {$4$};
    
            \draw[edge] (a)--(c);
            \draw[edge] (b)--(x)--(y);
            \draw[edge] (c)--(g)--(d)--(e);
            \path (b) edge [dotted, very thick] node{$\ell-2$} (c);

            \path (d) edge [bend left = 10, dotted, very thick, ->, color = blue!50] node{$1$} (b);
            \path (e) edge [bend right, dotted, very thick, ->, color = blue!50] node{$2$} (a);
            \path (b) edge [bend right = 80, dotted, very thick, ->, color = blue!50] node{$3$} (e);
            \path (a) edge [bend left = 80, dotted, very thick, ->, color = blue!50] node{$4$} (d);
        \end{tikzpicture}
        \caption{If $3$ is blocking some person, we can swap them.}
        \label{fig_19}
    \end{figure}
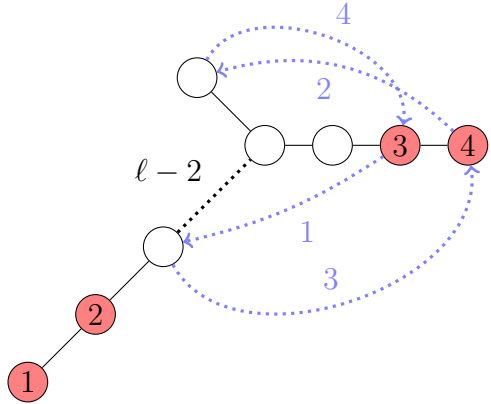
    
    Thus, we can now assume $\ell = 2$. By assumption, the branch from $y_1$ containing $u_1$ is also a path. Since the assumptions of \Cref{lem_tree_move_anywhere} still apply, we can move $1$ and $2$ down that path. Using the same argument as above, we can move $1$ and $2$ to occupy $u_1$ and $y_1$ and still have $k$ empty spots when the branch is cut off beyond $u_1$. 
    
    Therefore, we can assume that both $u_1$ and $x_1$ are leaves. Now if we cut off $x_1$ with $1$ standing there, we still have a graph without $k$-bridge and at least $k$ empty spots. By \Cref{lem_tree_move_anywhere}, we can have $2$ and $3$ occupy $u_1$ and $y_1$. We can also move empty spots to occupy the rest of the dog bone $D$. If $2$ and $3$ are in the wrong order, we can swap them using the fork in the dog bone. Thus, we have arrived at the configuration from the claim and we are done. 

    \textbf{Case IV: $X$ is a tree with a single vertex of degree $3$ and $X \neq T_6, T_7, T_8$.} Suppose $X$ consists of three paths $a_1, \ldots, a_{\ell_1}$, $b_1, \ldots, b_{\ell_2}$, and $c_1, \ldots, c_{\ell_3}$ with $a_{\ell_1} = b_{\ell_2} = c_{\ell_3}$ being the vertex of degree $3$. Then $\ell_i \leq k$ for each $i$ and $\ell_1 + \ell_2 + \ell_3 = n + 2$. Without loss of generality, assume that $\ell_1 \geq \ell_2 \geq \ell_3$. Since $X$ is not one of the exceptions, $\ell_2 \geq k \geq 4$. 
    
    By \Cref{lem_tree_move_anywhere}, we can perform a sequence of friendly swaps without the swap across $\{1, 2\}$ to bring $1, 2$ to occupy $a_1, a_2$. Suppose without loss of generality that $1$ is at $a_1$. Since the graph $X \setminus a_1$ has no $k$-bridge and currently contains at least $k$ empty spots, we can use \Cref{lem_tree_move_anywhere} to bring $2, 3$ to occupy $a_2, a_3$. Notice that there are two possible orderings: $1, 2, 3$ or $1, 3, 2$ at $a_1, a_2, a_3$ respectively. 

    Now we would like to move $1, 2, 3$ to occupy $a_{\ell_1 - 3}, a_{\ell_1 - 2}, a_{\ell_1 - 1}$ and have $a_1, \ldots, a_{\ell_1 - 4}$ occupied by other people (elements in $\{1, \ldots, k\}$). Since there are at least $k$ empty spots, we can move them to occupy $a_4, \ldots, a_{\ell_1}, b_{\ell_2 - 1}, b_{\ell_2 - 2}, b_{\ell_2 - 3}$. We can also move the next person we need to bury to $c_{\ell_3}$. Then we can move $1, 2, 3$ to $b_{\ell_2 - 1}, b_{\ell_2 - 2}, b_{\ell_2 - 3}$ and bury the person at $c_{\ell_3 - 1}$ at $a_1$ (see \Cref{fig_21}). Moving $1, 2, 3$ back, it is now as if we are in the case where $\ell_1$ is smaller by $1$. Induction then finishes this procedure. 

    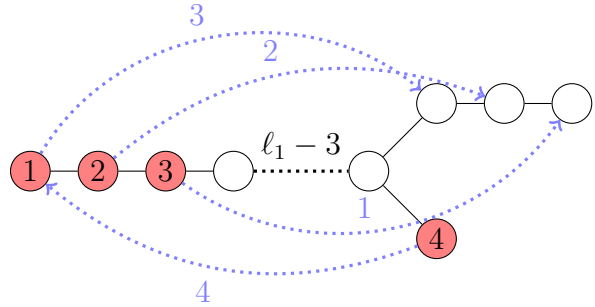
\begin{figure}[ht]
        \centering
        \begin{tikzpicture}[node distance = 2 cm, auto, scale = 0.9]
            \tikzstyle{spot} = [circle, draw, inner sep = 0pt]
            \tikzstyle{edge} = [-]
            \tikzstyle{person} = [circle, draw, fill = red!50, inner sep = 0pt]

            \node[person](a) at (0, 0) {\phantom{$f_1$}};
            \node[person](b) at (1, 0) {\phantom{$f_1$}};
            \node[person](c) at (2, 0) {\phantom{$f_1$}};
            \node[spot](d) at (3, 0) {\phantom{$f_1$}};
            \node[spot](d') at (5, 0) {\phantom{$f_1$}};
            \node[spot](e) at (6, 1) {\phantom{$f_1$}};
            \node[spot](f) at (7, 1) {\phantom{$f_1$}};
            \node[spot](g) at (8, 1) {\phantom{$f_1$}};
            \node[person](h) at (6, -1) {\phantom{$f_1$}};
    
            \node () at (a) {$1$};
            \node () at (b) {$2$};
            \node () at (c) {$3$};
            \node () at (h) {$4$};
    
            \draw[edge] (a)--(b)--(c)--(d);
            \draw[edge] (d')--(e)--(f)--(g);
            \draw[edge] (d')--(h);
            \path (d) edge [dotted, very thick] node{$\ell_1-3$} (d');

            \path (c) edge [bend right = 45, dotted, very thick, ->, color = blue!50] node{$1$} (g);
            \path (b) edge [bend left = 30, dotted, very thick, ->, color = blue!50] node{$2$} (f);
            \path (a) edge [bend left = 50, dotted, very thick, ->, color = blue!50] node{$3$} (e);
            \path (h) edge [bend left, dotted, very thick, ->, color = blue!50] node{$4$} (a);

        \end{tikzpicture}
        \caption{The way to bury the person at $c_{\ell_3 - 1}$. The order of $1, 2, 3$ may not be as shown in the figure. }
        \label{fig_21}
    \end{figure}
    
    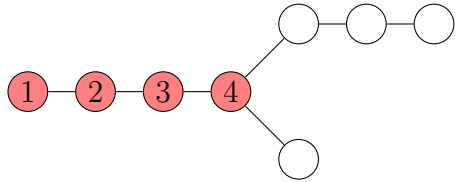
\begin{figure}[ht]
        \centering
        \begin{tikzpicture}[node distance = 2 cm, auto, scale = 0.9]
            \tikzstyle{spot} = [circle, draw, inner sep = 0pt]
            \tikzstyle{edge} = [-]
            \tikzstyle{person} = [circle, draw, fill = red!50, inner sep = 0pt]

            \node[person](a) at (0, 0) {\phantom{$f_1$}};
            \node[person](b) at (1, 0) {\phantom{$f_1$}};
            \node[person](c) at (2, 0) {\phantom{$f_1$}};
            \node[person](d) at (3, 0) {\phantom{$f_1$}};
            \node[spot](e) at (4, 1) {\phantom{$f_1$}};
            \node[spot](f) at (5, 1) {\phantom{$f_1$}};
            \node[spot](g) at (6, 1) {\phantom{$f_1$}};
            \node[spot](h) at (4, -1) {\phantom{$f_1$}};
    
            \node () at (a) {$1$};
            \node () at (b) {$2$};
            \node () at (c) {$3$};
            \node () at (d) {$4$};
    
            \draw[edge] (a)--(b)--(c)--(d)--(e)--(f)--(g);
            \draw[edge] (d)--(h);
        \end{tikzpicture}
        \caption{The people $1$ and $3$ are exchangeable from any configuration with $1, 2, 3, 4$ and four empty spots in this subgraph.}
        \label{fig_22}
    \end{figure}
    
    Now suppose $1$ is at $a_{\ell_1 - 3}$, $2, 3$ occupy $a_{\ell+1 - 2}, a_{\ell_1 - 1}$. Since $k \geq \ell_1$, we can bring a person (without loss of generality, say it is $4$) to $a_{\ell_1}$. We can also bring empty spots to $b_{\ell_2 - 1}, b_{\ell_2 - 2}, b_{\ell_2 - 3}, c_{\ell_3 - 1}$. Let $T$ be the induced subgraph of $X$ by the set of eight vertices $\{a_{\ell_1 - 3},a_{\ell_1 - 2},a_{\ell_1 - 1},a_{\ell_1},b_{\ell_2 - 1}, b_{\ell_2 - 2}, b_{\ell_2 - 3}, c_{\ell_3 - 1}\}$ (see \Cref{fig_22}). Let $K''$ be the graph $K_{4, 4}$ with an additional edge between $1$ and $2$. By brute force, we can check that $\FS(T, K'')$ is connected. In particular, $1$ and $3$ are $(T, K'')$-exchangeable from this configuration. By \Cref{lem_u_v_equivalent} and \Cref{lem_localization}, $1$ and $3$ are $(X, K')$-exchangeable from $\sigma$. 
\end{proof}

\section{Concluding remarks and future directions}\label{sec_conclusion}
In \Cref{thm_main}, we characterized when $\FS(X, K_{k_1, \ldots, k_t})$ is connected. The two ``non-trivial'' conditions that appear are that $X$ must not contain an $(n - k_t)$-bridge and that $X$ must be non-bipartite when $t = 2$. We explored what happens when we drop the non-bipartite condition in the case $t = 2$ in \Cref{thm_two_comps}. A natural question to ask is what happens if $X$ does contain an $(n - k_t)$-bridge. 
\begin{question}\label{qst_bridge}
    Suppose $t \geq 2$ and $1 \leq k_1 \leq \cdots \leq k_t$, where $n = k_1 + \cdots + k_t \geq 4$ and $k_t < n - 1$. Let $X$ be a graph on $n$ vertices that is connected, is not a cycle, and is non-bipartite if $t = 2$. Suppose that the length of the longest bridge in $X$ is $n - k_t$. What are the possible numbers of connected components in $\FS(X, K_{k_1, \ldots, k_t})$ and when is each achieved? 
\end{question}
In previous works, there are relatively few results on characterizing when a family of friends-and-strangers graphs have exactly $m$ connected components, where $m > 2$. Thus, answering \Cref{qst_bridge}, even just in the case $t = k_1 = 2$, is already an interesting avenue for future research. 

\bibliographystyle{amsplain0}
\bibliography{bib}
\end{document}